\documentclass[12pt]{article}
\usepackage{latexsym, amsmath, amsfonts, amsthm, amssymb}
\usepackage{times}
\usepackage{a4wide}
\usepackage{tikz}
\usetikzlibrary{matrix,arrows}

\def \Vol {\rm Vol}

\def \susbeteq {\subseteq}

\def \id {{\rm Id}}
\def \graph {{\rm Graph}}
\def \epigraph {{\rm Epigraph}}
\def \supp {{\rm Supp}}
\def \conv {{\rm Conv}}
\def \dom {{\rm Dom}}
\def \bary {{\rm bar}}

\def \RR {\mathbb R}

\def \EE {\mathbb E}

\def \eps {\varepsilon}
\def \vphi {\varphi}

\def \cH {\mathcal H}

\def \cM {\mathcal M}

\def \cI {\mathcal I}

\newtheorem{theorem}{Theorem}[section]
\newtheorem{lemma}[theorem]{Lemma}
\newtheorem{proposition}[theorem]{Proposition}
\newtheorem{corollary}[theorem]{Corollary}
\newtheorem{remark}[theorem]{Remark}

\newtheorem{definition}[theorem]{Definition}

\def\myffrac#1#2 in #3{\raise 2.6pt\hbox{$#3 #1$}\mkern-1.5mu\raise 0.8pt\hbox{$
#3/$}\mkern-1.1mu\lower 1.5pt\hbox{$#3 #2$}}

\def\qed{\hfill $\vcenter{\hrule height .3mm
\hbox {\vrule width .3mm height 2.1mm \kern 2mm \vrule width .3mm
height 2.1mm} \hrule height .3mm}$ \bigskip}

\begin{document}

\title{Affine hemispheres of elliptic type}
\date{}
\author{Bo'az Klartag\thanks{School of Mathematical Sciences, Tel Aviv University, Tel Aviv 69978, Israel. E-mail: klartagb@tau.ac.il  Supported by a grant from the European Research Council.}}

\maketitle

\begin{abstract}
We find that for any
 $n$-dimensional,
compact, convex set $K \subseteq \RR^{n+1}$ there is an affinely-spherical hypersurface $M \subseteq \RR^{n+1}$ with center at the relative interior of
$K$,
such that the disjoint union $M \cup K$ is the boundary of an $(n+1)$-dimensional, compact, convex set. This
so-called affine hemisphere $M$
is uniquely determined by $K$ up to affine transformations, it is of elliptic type, is associated with $K$ in an affinely-invariant manner,
and it is centered at the Santal\'o point of $K$.
\end{abstract}

\section{Introduction}
\setcounter{equation}{0}
\label{sec_intro}

Let $M \subseteq \RR^{n+1}$ be a smooth, connected hypersurface which
is locally strongly-convex, i.e.,
the second fundamental form is a definite symmetric bilinear form at any point $y \in M$.
There are several ways to define the affine normal line $\ell_M(y)$ at a point $y \in M$. One possibility is
to define $\ell_M(y)$ via the following procedure:

\begin{enumerate}
\item[(i)] Let $H = T_y M$ be the tangent space to $M$ at the point $y \in M$, viewed as a linear subspace of codimension one in $\RR^{n+1}$.
Select a vector $v \not \in H$ pointing to the convex side of $M$ at the point $y \in M$, and denote $ M_t = M \cap (H + t v)$ for $t > 0$. Here, $H + t v = \{ x + t v \, ; \, x \in H \}$.
\item[(ii)] For a sufficiently small $t >0 $, the section $M_t$ encloses an $n$-dimensional convex body $\Omega_t \subseteq H + t v$.
The barycenters $b_t = \bary(\Omega_t)$ depend smoothly on $t$. The affine normal line $\ell_M(y) \subseteq \RR^{n+1}$ is defined to be the line passing through $y$ in the direction of
the non-zero vector $\left. \frac{d}{dt} b_t \right|_{t = 0}$.
\end{enumerate}

We say that $M$   is {\it affinely-spherical} with center at a point $p \in \RR^{n+1}$ if all of the affine normal lines of $M$ meet at $p$.
In the case where all of the affine normal lines are parallel, we say that $M$ is affinely-spherical with center at infinity.
An {\it affine sphere} is an affinely-spherical hypersurface which is {\it complete}, i.e., it is a closed subset of $\RR^{n+1}$.
This definition is clearly affinely-invariant, hence the term ``affine sphere''.
In Section \ref{sec_5} below we explain that $M$ is affinely-spherical with center at the origin if and only if the cone measure on $M$ is mapped to a measure proportional
to the cone measure
on the polar hypersurface $M^*$ via the polarity map.

\medskip Affine spheres were introduced by the Romanian geometer Tzitz\'eica \cite{Tz1, Tz2}.
All convex quadratic hypersurfaces in $\RR^{n+1}$ are affine spheres, as well as the hypersurface
$$ M = \left \{ (x_1,\ldots,x_n) \in \RR^n \, ; \, \forall i, x_i > 0, \, \prod_{i=1}^n x_i = 1 \right \}, $$
found by Tzitz\'eica \cite{Tz1, Tz2} and Calabi \cite{Cal2}. See Loftin \cite{loft} for a survey on affine spheres.
At any point $y \in M$, the punctured line $\ell_M(y) \setminus \{ y \} $ is naturally divided into two rays: one pointing to
the convex side of $M$ and the other to the concave side. These two rays are referred to as the convex side and the concave side of $\ell_M(y)$, respectively.
An affinely-spherical hypersurface $M$ is called {\it elliptic} if
its center lies on the convex side of all of the affine normal lines.
It is {\it hyperbolic} if its center lies on the concave
side of all of the affine normal lines. There are also parabolic affine spheres, whose affine normal lines are all parallel.

\medskip Ellipsoids in $\RR^{n+1}$ are elliptic affine spheres, while elliptic paraboloids are parabolic affine spheres.
There are no other examples of complete affine spheres of elliptic or parabolic type.  This non-trivial theorem
is the culmination of the works of Blaschke \cite{Bla}, Calabi \cite{Cal1}, Pogorelov \cite{Pog} and Trudinger and Wang \cite{Tru}.

\medskip While affine spheres of elliptic or parabolic type are quite rare, there are
many hyperbolic affine spheres in $\RR^{n+1}$. From the works of Calabi \cite{Cal2}  and Cheng-Yau \cite{CY} we learn that for any non-empty, open, convex cone $C \subseteq \RR^{n+1}$
that does not contain a full line, there exists a hyperbolic affine sphere which is asymptotic to the cone.
This hyperbolic affine sphere is determined by the cone $C$ up to homothety, and all
hyperbolic affine spheres in $\RR^{n+1}$ arise this way. Why are there so few elliptic affine spheres,
compared to the abundance of hyperbolic affine spheres? Perhaps completeness is  too strong a requirement in the elliptic case. We propose the following:

\begin{definition} Let $M  \subseteq \RR^{n+1}$ be a smooth, connected, locally strongly-convex hypersurface.
We say that $M$ is an ``affine hemisphere'' if
\begin{enumerate}
\item There exist compact, convex sets $K, \tilde{K} \subseteq \RR^{n+1}$, with $\dim(K) = n$ and $\dim(\tilde{K}) = n+1$, such that
$M$ does not intersect the affine hyperplane spanned by $K$ and
$$ K \cup M = \partial \tilde{K}. $$
\item The hypersurface $M$ is affinely-spherical with center at the relative interior of $K$.
\end{enumerate}
We say that $K$ is the ``anchor'' of the affine hemisphere $M$.
\label{def_1251}
\end{definition}

In Definition \ref{def_1251},
the dimension $\dim(K)$ is the maximal number $N$
such that $K$ contains $N+1$ affinely-independent vectors.
Note that when $M \susbeteq \RR^{n+1}$ is an affine hemisphere, its anchor $K$ is the compact, convex
set enclosed by $\overline{M} \setminus M$, where $\overline{M}$ is the closure of $M$.
In particular, $ K = \conv(\overline{M} \setminus M) $
where $\conv$ denotes convex hull.
It is clear that an affine hemisphere is always of elliptic type.

\begin{theorem} Let $K \subseteq \RR^{n+1}$ be an $n$-dimensional, compact, convex set. Then there exists
an affine hemisphere $M \subseteq \RR^{n+1}$ with anchor $K$, uniquely determined up to affine transformations.
The affine hemisphere $M$ is centered at the Santal\'o point of $K$.
\label{thm1}
\end{theorem}

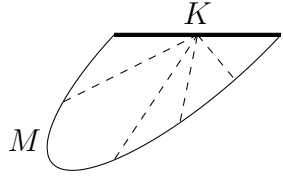
\begin{figure}
\begin{center}
\begin{tikzpicture}[scale = 2]
\draw[ samples=300, domain=-1:-0.555] plot (\x, {   0.5 * (1.5*\x + sqrt(1 - \x*\x))   });
\draw[samples=300, domain=-1:0.555] plot (\x, {   0.5 * (1.5*\x - sqrt(1 - \x*\x))   });
\draw [ultra thick] (-0.555,0) -- (0.555,0);
\draw[dashed, domain=-0.89:0] plot (\x, {\x / 2});
\draw[dashed, domain=0:0.25] plot (\x, {-\x * 1.2 });
\draw[dashed, domain=-0.115:0] plot (\x, {\x * 5 });
\draw[dashed, domain=-0.55:0] plot (\x, {\x *1.5 });
\node [above] at (0,0) {$K$};
\node at (-1.15,-0.7) {$M$};
\end{tikzpicture}
\caption{Half of an ellipse, which  is an affine one-dimensional hemisphere in $\RR^2$. \label{fig1}}
\end{center}
\end{figure}

Thus, with any $n$-dimensional, compact, convex set $K \subseteq \RR^{n+1}$ we associate an $(n+1)$-dimensional, compact, convex set $\tilde{K} \subseteq \RR^{n+1}$
whose boundary consists of two parts: the convex set $K$ itself is a facet, and the rest
of the boundary is an affine hemisphere $M$ centered at the Santal\'o point of $K$.
We refer the reader to Loftin \cite{loft} and to Nomizu and Sasaki \cite{NS} for information about the rich geometric structure associated with
affinely-spherical hypersurfaces. Let us just observe here that
by \cite[Theorem 6.5]{NS}, any affine function in $\RR^{n+1}$ that vanishes on $K$ is an eigenfunction of the affine-metric Laplacian of $M$
with Dirichlet boundary conditions, corresponding to the first eigenvalue.


\medskip The proof of Theorem \ref{thm1} is basically a variant of the {\it moment measure} construction by
Cordero-Erausquin and the author \cite{CK} which is in turn  influenced  by Berman and Berndtsson \cite{BB}
and is also analogous to the classical Minkowski problem.
Let us now present a few questions about affine hemispheres:

\begin{enumerate}
\item Other than half-ellipsoids, we are not aware of any affine hemisphere
that may be described by a simple formula. Is there a closed form
for the affine hemisphere associated with the $n$-dimensional simplex or the $n$-dimensional cube? For  moment measures, the solutions in the case of the simplex and the cube are given by explicit formul\ae,
see \cite{CK}.
\item  Calabi \cite{Cal2} found a composition rule
for hyperbolic affine
spheres, allowing one to construct a
hyperbolic affine
sphere of dimension $n+m+1$ from two hyperbolic affine spheres of dimensions $n$ and $m$.
Is there an analogous construction for affine hemispheres?
\item
An intriguing question is whether an affine hemisphere $M$ can be extended
beyond its anchor $K$, to an affinely-spherical hypersurface $\tilde{M} \supsetneq M$.
When the anchor $K$ is an ellipsoid, the affine hemisphere $M$ with anchor $K$ is half an ellipsoid, and
may clearly be extended to the surface of a full ellipsoid. On the other hand,
if $K$ is a polytope, then the affine hemisphere $M$ cannot be smoothly extended beyond the vertices of $K$.
\item Finally, is there a theory similar  to that of affine hemispheres that is related to {\it parabolic} affinely-spherical hypersurfaces?
See Ferrer, Mart\'inez and Mil\'an \cite{FMM}, Mil\'an \cite{milan} and Remark \ref{rem_1338} below for partial results in this direction. 

\end{enumerate}

\medskip
Throughout this paper, by smooth we always mean $C^{\infty}$-smooth.
We write $| \cdot |$ for the usual Euclidean norm in $\RR^n$, and $S^{n} = \{ x \in \RR^{n+1} \, ; \, |x| = 1 \}$
is the Euclidean unit sphere centered at the origin. The standard scalar product of $x, y \in \RR^n$ is denoted by $\langle x, y \rangle$.
We write  $\log$ for the natural logarithm. For a Borel measure $\mu$ in $\RR^n$ we denote by
$\supp(\mu)$ the support of $\mu$, which is the intersection of all closed sets of a full $\mu$-measure.
A hypersurface in $\RR^{n+1}$ is an $n$-dimensional submanifold of $\RR^{n+1}$.
A submanifold $M \subseteq \RR^{n+1}$ {\it encloses} a convex set $K \subseteq \RR^{n+1}$ if $M$ is the boundary
of $K$ relative to the affine subspace spanned by $K$.

\medskip
\emph{Acknowledgements.} Let me express my gratitude to Bo Berndtsson, Ronen Eldan and Yanir Rubinstein
for interesting discussions and for explanations and references on affine differential geometry.

\section{A variational problem}
\setcounter{equation}{0}
\label{sec_var}

In this section we analyze a variational problem related to affine hemispheres. Similar variational problems were
 considered by Berman and Berndtsson \cite{BB} and by Cordero-Erausquin and the author
\cite{CK}. For a function $\psi: \RR^n \rightarrow \RR \cup \{ + \infty \}$ denote
$$ \dom(\psi) = \left \{ x \in \RR^n \, ; \, \psi(x) < +\infty \right \}. $$
The Legendre transform of $\psi$ is the convex function
$$ \psi^*(y) = \sup_{x \in \dom(\psi)} \left[ \langle x, y \rangle - \psi(x) \right] \qquad \qquad \qquad (y \in \RR^n), $$
where $\sup \emptyset = -\infty$. The function $\psi^*$ is always convex and lower semi-continuous.
A convex function $\psi: \RR^n \rightarrow \RR \cup \{ + \infty \}$ is {\it proper} if it is lower semi-continuous
with $\dom(\psi) \neq \emptyset$. 
When $\psi$ is convex and proper, the Legendre transform $\psi^*$ is again convex and proper, and $\psi^{**} = \psi$.
We will frequently use the formula $\psi^*(0) = -\inf \psi$, as well as the relation $(\lambda \psi)^*(x) = \lambda \psi^*(x/\lambda)$,
which is valid for any $x \in \RR^n$ and $\lambda > 0$. It is also well-known  that for any $v \in \RR^n$, denoting $\psi_1(x) = \psi(x) + \langle x, v \rangle$,
\begin{equation}  \psi_1^*(y) = \psi^*(y - v) \qquad \qquad \qquad (y \in \RR^n). \label{eq_1021} \end{equation}
See Rockafellar \cite{roc} for a thorough discussion of the Legendre transform. For $p > 0$ and a function $\psi: \RR^n \rightarrow \RR \cup \{ + \infty \}$ with $\psi(0) < 0$ we define
\begin{equation}  \cI_p(\psi) = \left( \int_{\RR^n} \frac{dx}{(\psi^*(x))^{n+p}} \right)^{-1/p} \in [0, +\infty]. \label{eq_302}
\end{equation}
Two remarks are in order: First, note that $\inf \psi^* \geq - \psi(0) > 0$, and that the integral in (\ref{eq_302}) is a well-defined
element of $[0, +\infty]$. Second, for the purpose
of  definition (\ref{eq_302}) let us agree that $0^{-\alpha} = +\infty$ and $(+\infty)^{-\alpha} = 0$ for $\alpha > 0$.
The functional $\cI_p$ is closely related to the Borell-Brascamp-Lieb inequality \cite{Bor, BL}.
The latter inequality, which is a variant of Brunn-Minkowski, states the following:
For any $0 < \lambda < 1$  and three convex functions $\vphi_{\lambda},\vphi_0,\vphi_1: \RR^n \rightarrow (0, +\infty]$ such that
\begin{equation} \vphi_{\lambda} \left( (1 - \lambda) x + \lambda y  \right) \leq (1 - \lambda) \vphi_0(x) + \lambda \vphi_1(y) \qquad \qquad \qquad (x, y \in \RR^n), \label{eq_1153}
\end{equation}
we have,
\begin{equation} \left( \int_{\RR^n} \frac{dx}{\vphi_{\lambda}(x)^{n+p}} \right)^{-1/p} \leq (1 - \lambda)
\left( \int_{\RR^n} \frac{dx}{\vphi_0(x)^{n+p}} \right)^{-1/p} + \lambda \left( \int_{\RR^n} \frac{dx}{\vphi_1(x)^{n+p}} \right)^{-1/p}. \label{eq_1155} \end{equation}
The Borell-Brascamp-Lieb inequality, sometimes called the dimensional Pr\'ekopa inequality, implies the convexity of $\cI_p$
as is stated in the following:

\begin{lemma} Let $p, \lambda > 0$, and let $\psi, \psi_0, \psi_1: \RR^n \rightarrow \RR \cup \{ + \infty \}$
be functions that are negative at zero. Denote $\vphi = \psi^*, \vphi_0 = \psi_0^*$ and $\vphi_1 = \psi_1^*$. Then the following hold:  \begin{enumerate} \item[(i)] $\displaystyle \cI_p(\lambda \psi) = \lambda \cI_p(\psi)$.
\item[(ii)] $\displaystyle \cI_p( \psi_0 + \psi_1 ) \leq \cI_p(\psi_0) + \cI_p(\psi_1)$.
\item[(iii)] Assume that $\dom(\vphi_0) = \dom(\vphi_1) = \RR^n$.
Then equality in (ii) holds if and only if there exist $x_0 \in \RR^n$ and $\lambda > 0$ such that
$$\vphi_1(x) = \lambda \vphi_0( x_0 +  x / \lambda ) \qquad \qquad \qquad \text{for all} \ x \in \RR^n.
 $$
\end{enumerate} \label{lem_1530}
\end{lemma}

\begin{proof} By using the formula $(\lambda \psi)^*(x) = \lambda \vphi(x/ \lambda)$, which is valid for any $x \in \RR^n$, we obtain
\begin{align*}
\cI_p(\lambda \psi) = \left( \int_{\RR^n} \frac{dx}{(\lambda \vphi(x / \lambda))^{n+p}} \right)^{-1/p} = \lambda^{\frac{n+p}{p}}
\cdot \lambda^{-\frac{n}{p}}  \left( \int_{\RR^n} \frac{dx}{\vphi(x)^{n+p}} \right)^{-1/p} = \lambda \cI_p(\psi).
\end{align*}
Thus (i) is proven. Next, denote  $\vphi_{1/2} = \left[ (\psi_0 + \psi_1)/2 \right]^*$. Then $\vphi_0, \vphi_1, \vphi_{1/2}: \RR^n \rightarrow (0, +\infty]$ are  convex functions, and for any $x, y \in \RR^n$,
\begin{align*}
\vphi_{1/2} & \left( \frac{x+y}{2} \right)  = \sup_{z \in \dom(\psi_0) \cap \dom(\psi_1) } \left[ \left \langle \frac{x+y}{2}, z \right\rangle - \frac{\psi_0(z) + \psi_1(z)}{2} \right]
\\ & \leq \frac{1}{2} \left \{ \sup_{z \in \dom(\psi_0)} \left[ \langle x, z \rangle - \psi_0(z) \right] + \sup_{z \in \dom(\psi_1)} \left[ \langle y, z \rangle - \psi_1(z) \right]
\right \} = \frac{\vphi_0(x) + \vphi_1(y)}{2}.
\end{align*}
Hence condition (\ref{eq_1153}) is satisfied, with $\lambda = 1/2$. The case $\lambda = 1/2$ of the Borell-Brascamp-Lieb inequality (\ref{eq_1155}) implies that
$$ \cI_p \left( \frac{\psi_0 + \psi_1}{2} \right) \leq \frac{\cI_p(\psi_0) + \cI_p(\psi_1)}{2} $$
and (ii) now follows from (i). According to Dubuc \cite{dubuc}, equality holds in (\ref{eq_1155}), with $\vphi_0, \vphi_1: \RR^n \rightarrow (0, +\infty)$ being convex functions,
if and only if there exist $\lambda > 0$ and $x_0 \in \RR^n$ such that
$\vphi_1(x) = \lambda \vphi_0( x_0 + x/\lambda )$ for all $x \in \RR^n$. This proves (iii).
\end{proof}

The next lemma describes a lower semi-continuity property of the functional $\cI_p$.

\begin{lemma} Let $p > 0$ and let $K \subseteq \RR^n$ be a convex, open set containing the origin.
Let $\psi: \RR^n \rightarrow \RR \cup \{ + \infty \}$ be a convex function with $\psi(0) < 0$ such that $K \subseteq \dom(\psi) \subseteq \overline{K}$.
Assume that for any $\ell \geq 1$ we are given a function $\psi_{\ell}: \RR^n \rightarrow \RR \cup \{ + \infty \}$ with $\psi_{\ell}(0) < 0$, such that $\psi_{\ell} \longrightarrow \psi$ pointwise in the set $K$ as $\ell \rightarrow \infty$. Then,
$$ \cI_p(\psi) \leq \liminf_{\ell \rightarrow \infty} \cI_p(\psi_\ell). $$
\label{lem_1402}
\end{lemma}

\begin{proof} The convex function $\psi$ is finite and hence continuous in the convex, open set $K$. Since $0 \in K$ and $\psi(0) < 0$, we may find $\eps > 0$ and linearly independent vectors $v_1,\ldots,v_n \in K$ such that
$$ \psi( \pm v_i ) < - \eps \qquad \text{for} \ i=1,\ldots, n. $$
By the pointwise convergence in $K$, there exists $\ell_0$ such that $\psi_{\ell}(\pm v_i) < -\eps$ for all $\ell \geq \ell_0$
and $i=1,\ldots,n$. The convex hull of the $2n$ points $\{ \pm v_i \, ; \, i=1,\ldots, n \}$ contains a Euclidean ball
of radius $\delta > 0$ centered at the origin. Consequently, for $\ell \geq \ell_0$ and $x \in \RR^n$,
\begin{equation} \psi_{\ell}^*(x) = \sup_{y \in \dom(\psi_{\ell})} \left[ \langle x, y \rangle -
\psi_{\ell}(x) \right] \geq \sup_{i=1,\ldots,n} \left[ \left| \langle x, v_i \rangle \right| + \eps \right ] \geq \eps + \delta |x|.
\label{eq_1219}
\end{equation}
Next,  we claim that for any $x_0 \in \RR^n$,
\begin{equation}
\psi^*(x_0) \leq \liminf_{\ell \rightarrow \infty} \psi_{\ell}^*(x_0). \label{eq_1030}
\end{equation}
Indeed, since $\psi$ is convex, its restriction to any line segment in the convex set $\dom(\psi)$ is upper semi-continuous (see, e.g., \cite{GKR}).
From the inclusion $\dom(\psi) \subseteq \overline{K}$ we thus learn that
$$ \psi^*(x_0) = \sup_{y \in \dom(\psi)} \left[ \langle x_0, y \rangle - \psi(y) \right] =
\sup_{y \in K} \left[ \langle x_0, y \rangle - \psi(y) \right]. $$
Hence, for any $\eps > 0$
there exists $y_0 \in K$ such that $\psi^*(x_0) \leq \eps + \langle x_0, y_0 \rangle - \psi(y_0)$.
By the pointwise convergence in $K$, for a sufficiently large $\ell$ we observe that $\psi_{\ell}(y_0) \leq \psi(y_0) + \eps$.
Therefore, for a sufficiently large $\ell$,
$$ \psi_{\ell}^*(x_0) \geq \langle x_0, y_0 \rangle - \psi_{\ell}(y_0) \geq -\eps + \langle x_0, y_0 \rangle - \psi(y_0) \geq -2 \eps + \psi^*(x_0) $$
and (\ref{eq_1030}) is proven. The function $(\eps + \delta |x|)^{-(n+p)}$ is integrable in $\RR^n$.
 Thanks to (\ref{eq_1219}) and (\ref{eq_1030})  we may use the dominated convergence theorem, and conclude that
\begin{align*} \nonumber \int_{\RR^n} \frac{dx}{\left( \psi^*(x) \right)^{n+p}} & \geq \int_{\RR^n} \left[ \lim_{\ell \rightarrow \infty} \sup_{k \geq \ell} \frac{1}{\left( \psi_{k} ^*(x) \right)^{n+p}} \right] dx = \lim_{\ell \rightarrow \infty}  \int_{\RR^n} \left[ \sup_{k \geq \ell} \frac{1}{\left( \psi_{k} ^*(x) \right)^{n+p}} \right] dx \\
& = \limsup_{\ell \rightarrow \infty}  \int_{\RR^n} \left[ \sup_{k \geq \ell} \frac{1}{\left( \psi_{k} ^*(x) \right)^{n+p}} \right] dx
\geq \limsup_{\ell \rightarrow \infty}  \int_{\RR^n} \frac{dx}{\left( \psi_{\ell} ^*(x) \right)^{n+p}}.
\tag*{\qedhere}
\end{align*}
\end{proof}

The next theorem is our main result in this section.
It is essentially a theorem about the Legendre transform of the functional $\cI_p^2$, viewed as a convex functional on an infinite-dimensional cone.

 \begin{theorem} Let $p > 0$ and let $\mu$ be a Borel probability measure on $\RR^n$ with $\int_{\RR^n} |x| d \mu(x) < +\infty$ such that the barycenter of
$\mu$ lies at the origin. Assume that the origin belongs to the interior of $\conv(\supp(\mu))$.
Then there exists a $\mu$-integrable, proper, convex function $\psi: \RR^n \rightarrow \RR \cup \{ + \infty \}$ with $\psi(0) < 0$
such that
\begin{equation} \int_{\RR^n} \psi d \mu + \left( \int_{\RR^n} \frac{dx}{(\psi^*(x))^{n+p}} \right)^{-2/p}
\leq \int_{\RR^n} \psi_1 d \mu + \left( \int_{\RR^n} \frac{dx}{(\psi_1^*(x))^{n+p}} \right)^{-2/p} \label{eq_1046} \end{equation}
for any $\mu$-integrable, proper, convex function $\psi_1: \RR^n \rightarrow \RR \cup \{ + \infty \}$ with $\psi_1(0) < 0$.
Moreover, the expression on the left-hand side of (\ref{eq_1046}) is a finite, negative number, and $\psi(x) = +\infty$ for
any $x \in \RR^n \setminus \overline{K}$ where $K$ is the interior of $\conv(\supp(\mu))$.
\label{cor_515}
\end{theorem}

The remainder of this section is dedicated to the proof of Theorem \ref{cor_515}.
Let us fix a number $p > 0$ and a Borel probability measure $\mu$ satisfying the requirements
of Theorem \ref{cor_515}. For a $\mu$-integrable, proper convex function $\psi: \RR^n \rightarrow \RR \cup \{ + \infty \}$
with $\psi(0) < 0$ we denote
$$ \cI_{\mu, p}(\psi) = \int_{\RR^n} \psi d \mu + \cI_p^2(\psi) = \int_{\RR^n} \psi d \mu + \left( \int_{\RR^n} \frac{dx}{(\psi^*(x))^{n+p}} \right)^{-2/p}. $$
Since the barycenter of $\mu$ is at the origin, we learn from (\ref{eq_1021}) that
$\cI_{\mu, p}(\psi) = \cI_{\mu, p}(\psi_1)$
whenever $\psi_1(x) = \psi(x) + \langle x, v \rangle$ for some $v \in \RR^n$.
The first step in the proof of Theorem \ref{cor_515} is the following proposition:

\begin{proposition} Let $p > 0$ and let $\mu$ be as in Theorem \ref{cor_515}. Then,
$$  \inf_{\psi} \cI_{\mu, p}(\psi) > -\infty $$
where the infimum runs over all $\mu$-integrable, proper convex functions $\psi: \RR^n \rightarrow \RR \cup \{ + \infty \}$ with $\psi(0) < 0$.
\label{prop_515}
\end{proposition}

The proof of Proposition \ref{prop_515} relies on several lemmas.

\begin{lemma} There exist $c_1, c_2 > 0$, depending on $\mu$, with the following property: For any $\theta \in S^{n-1}$,
$$ \int_{\RR^n} \langle x, \theta \rangle 1_{\{ \langle x, \theta \rangle > c_1 \}} d \mu(x) \geq c_2, $$
where $1_{\{ \langle x, \theta \rangle > c_1 \}}$ equals one when $\langle x, \theta \rangle > c_1$ and it vanishes elsewhere.
\label{lem_948}
\end{lemma}

\begin{proof} The origin belongs to the interior of $\conv(\supp(\mu))$. Therefore, for any $\theta \in S^{n-1}$,
\begin{equation}
\int_{\RR^n} \langle x, \theta \rangle 1_{\{ \langle x, \theta \rangle > 0 \}} d \mu(x) > 0.
\label{eq_1340}
\end{equation}
For $t > 0$  consider the non-negative function
$$ f_t(\theta) = \int_{\RR^n} \langle x, \theta \rangle 1_{\{ \langle x, \theta \rangle > t \}} d \mu(x) \qquad \qquad \qquad (\theta \in S^{n-1}). $$
We claim that $f_t$ is lower semi-continuous. Indeed, if $\theta_j \longrightarrow \theta$ then by Fatou's lemma,
$$ f_t(\theta) = \int_{\RR^n} \langle x, \theta \rangle 1_{\{ \langle x, \theta \rangle > t \}} d \mu(x) \leq
\liminf_{j \rightarrow \infty} \int_{\RR^n} \langle x, \theta_j \rangle 1_{\{ \langle x, \theta_j \rangle > t \}} d \mu(x) = \liminf_{j \rightarrow \infty} f_t(\theta_j). $$
Denote by $m_t$ the minimum of the function $f_t$ on $S^{n-1}$, and let $\theta_t \in S^{n-1}$
be a point such that $f_t(\theta_t) = m_t$. Since $S^{n-1}$ is compact, there exists a sequence $t_j \rightarrow 0^+$ such that $\theta_{t_j} \rightarrow \theta$ for a certain unit
vector $\theta \in S^{n-1}$. By (\ref{eq_1340}) and Fatou's lemma,
$$ 0 < \int_{\RR^n} \langle x, \theta \rangle 1_{\{ \langle x, \theta \rangle > 0 \}} d \mu(x) \leq
\liminf_{j \rightarrow \infty} \int_{\RR^n} \langle x, \theta_{t_j} \rangle 1_{\{ \langle x, \theta_j \rangle > t_j \}} d \mu(x) = \liminf_{j \rightarrow \infty} m_{t_j}. $$
Consequently there exists $j \geq 1$ such that $m_{t_j} > 0$. The lemma follows with $c_1 = t_j$ and $c_2 = m_{t_j}$.
\end{proof}

\begin{lemma} There exists $c > 0$, depending on $\mu$, with the following property:
Let $\psi: \RR^n \rightarrow \RR \cup \{ + \infty \}$ be a proper, convex function that is $\mu$-integrable.
Denote $\alpha = -\psi(0)$. Assume that $\psi(0) = \inf \psi$ and that $\int_{\RR^n} \psi d \mu < 0$.  Then for any $x \in \RR^n$,
$$ \psi(x) \leq -\alpha/2 \qquad \text{when} \ |x| < c. $$
\label{lem_1020}
\end{lemma}

\begin{proof} We will prove the lemma with $c = \min \{ c_1, c_2/4 \}$ where $c_1, c_2$ are the positive constants from Lemma \ref{lem_948}.
Assume by contradiction that the conclusion of the lemma fails. Then the convex set  $A = \{ x \in \RR^n \, ; \, \psi(x) \leq -\alpha/2 \}$
does not contain an open ball of radius $c$ around the origin. By the convexity of $A$, there exists $\theta \in S^{n-1}$ such that $\langle x, \theta \rangle
< c$ for all $x \in A$. By the convexity of the function $\psi$, for any $x \in \RR^n$ with $\langle x, \theta \rangle \geq c$,
$$ -\frac{\alpha}{2} < \psi \left( \frac{c x}{\langle x, \theta \rangle} \right) \leq \frac{c}{\langle x, \theta \rangle} \psi(x)
+ \left(1 - \frac{c}{\langle x, \theta \rangle} \right) \psi(0) = \frac{c}{\langle x, \theta \rangle} \psi(x) - \alpha \cdot \left(1 - \frac{c}{\langle x, \theta \rangle} \right).
$$
Consequently, $\psi(x) \geq \alpha \langle x, \theta \rangle / (2c) - \alpha$ for any $x \in \RR^n$ with $\langle x, \theta \rangle \geq c$.
Since $\inf \psi = -\alpha$ and $c \leq c_1$, then by Lemma \ref{lem_948},
\begin{align*} \int_{\RR^n} & \psi d \mu  = \int_{\RR^n} \psi(x) 1_{\{ \langle x, \theta \rangle \leq c_1 \}} d \mu(x) +
\int_{\RR^n} \psi(x) 1_{\{ \langle x, \theta \rangle > c_1 \}} d \mu(x) \\ & \geq -\alpha + \int_{\RR^n}
\left[ \frac{\alpha}{2c} \cdot \langle x, \theta \rangle - \alpha \right] \cdot 1_{\{ \langle x, \theta \rangle > c_1 \}} d \mu(x) \geq -2\alpha + \frac{\alpha}{2c} \cdot c_2
\geq -2 \alpha + 2 \alpha = 0, \end{align*}
in contradiction to our assumption that $\int_{\RR^n} \psi d \mu < 0$.
\end{proof}

\begin{lemma} There exists $\tilde{c} > 0$, depending on $\mu$ and $p$, with the following property:
Let $\psi: \RR^n \rightarrow \RR \cup \{ + \infty \}$ be a proper, convex function that is $\mu$-integrable.
Denote $\alpha = -\psi(0)$. Assume that $\psi(0) = \inf \psi$ and that $\int_{\RR^n} \psi d \mu < 0$. Then,
$$ \cI_{\mu, p}(\psi) \geq -\alpha + \tilde{c} \alpha^2. $$
\label{lem_406}
\end{lemma}

\begin{proof}
From Lemma \ref{lem_1020}, for any $y \in \RR^n$,
$$ \psi^*(y) = \sup_{x \in \dom(\psi)} \left[ \langle x, y \rangle - \psi(x) \right] \geq
\sup_{x \in \RR^n, |x| < c} \left[ \langle x, y \rangle + \alpha/2 \right] = \frac{\alpha}{2} + c |y|. $$
Since $\inf \psi = -\alpha$, we deduce that
\begin{align*} \cI_{\mu, p}(\psi) & = \int_{\RR^n} \psi d \mu + \left( \int_{\RR^n} \frac{dy}{(\psi^*(y))^{n+p}} \right)^{-2/p} \nonumber
\geq -\alpha + \left( \int_{\RR^n} \frac{dy}{(\alpha/2 + c |y| )^{n+p}} \right)^{-2/p} \\ & = -\alpha + \alpha^{2}
\left( \int_{\RR^n} \frac{dy}{(1/2 + c |y| )^{n+p}} \right)^{-2/p} = -\alpha + \tilde{c} \alpha^{2}. \tag*{\qedhere}
\end{align*}
\end{proof}

\begin{lemma} Assume that $\psi: \RR^n \rightarrow \RR \cup \{ + \infty \}$ is a $\mu$-integrable, convex function.
Then $\dom(\psi)$ contains the interior of $\conv(\supp(\mu))$. In particular, $\dom(\psi)$ contains
the origin in its interior. \label{lem_1340}
\end{lemma}

\begin{proof} Otherwise, we could use a hyperplane and separate the convex set $\dom(\psi)$ from an open ball intersecting $\supp(\mu)$.
This would imply that $\psi$ is not $\mu$-integrable, in contradiction.
\end{proof}

\begin{proof}[Proof of Proposition \ref{prop_515}]
Let $\psi: \RR^n \rightarrow \RR \cup \{ + \infty \}$ be a proper, convex function with $\psi(0) < 0$ that is $\mu$-integrable.
We will show that
\begin{equation}  \cI_{\mu, p}(\psi) \geq - \frac{1}{4 \tilde{c}} \label{eq_427}
\end{equation}
where $\tilde{c} > 0$ is the constant from Lemma \ref{lem_406}.
 In the case where $\int \psi d \mu \geq 0$ we have $ \cI_{\mu, p}(\psi) \geq 0$, and (\ref{eq_427}) trivially holds.
We may thus assume that
\begin{equation} \int_{\RR^n} \psi d \mu < 0. \label{eq_1631}
\end{equation}
The origin is in the interior
of $\dom(\psi)$, according to Lemma \ref{lem_1340}. From Rockafellar \cite[Theorem 23.4]{roc}  we learn that there exists $w \in \RR^n$ such that
\begin{equation} \psi(x) \geq \psi(0) + \langle x, w \rangle \qquad \qquad \qquad (x \in \RR^n). \label{eq_922} \end{equation}
Recall that $\cI_{\mu, p}(\psi) = \cI_{\mu, p}(\psi_1)$
whenever $\psi_1(x) = \psi(x) + \langle x, v \rangle$ for some $v \in \RR^n$.
By adding an appropriate linear functional to $\psi$,
we may assume that $w = 0$ in (\ref{eq_922}) and hence $\psi(0) = \inf \psi$.
Denote $\alpha = -\psi(0)$, which is a positive number, as follows from (\ref{eq_1631}). We may now apply Lemma \ref{lem_406} and obtain that
\begin{equation*} \cI_{\mu, p}(\psi) \geq -\alpha + \tilde{c} \alpha^2 \geq -\frac{1}{4 \tilde{c}}, \end{equation*}
completing the proof of (\ref{eq_427}). The proposition is thus proven.
\end{proof}

The next proposition is the second step in the proof of Theorem \ref{cor_515}.

\begin{proposition} The infimum in Proposition \ref{prop_515} is attained.
\label{prop_340}
\end{proposition}

Again, the proof of Proposition \ref{prop_340} relies on a few small lemmas.

\begin{lemma} There exists a $\mu$-integrable, proper convex function $\psi: \RR^n \rightarrow \RR \cup \{ + \infty \}$ with $\psi(0) < 0$ such that
$\cI_{\mu, p}(\psi) < 0$. \label{lem_351}
\end{lemma}

\begin{proof} Let $\delta > 0$ and denote $\psi_{\delta}(x) = -\delta + \eps |x|$ for $\eps = \delta^{1 + p / (4n)}$. Then,
$$ \left( \int_{\RR^n} \frac{dx}{(\psi_{\delta}^*(x))^{n+p}} \right)^{-2/p} = \left( \int_{B(0, \eps)} \frac{dx}{\delta^{n+p}} \right)^{-2/p} = A \delta^{3/2}
$$
where $B(0, \eps) = \{ x \in \RR^n \, ; \, |x| < \eps \}$ and $A = \Vol_n(B(0,1))^{-2/p} > 0$.
Consequently,
$$ \cI_{\mu,p}(\psi_{\delta}) = A \delta^{3/2} + \int_{\RR^n} (-\delta + \eps |x|) d \mu(x) = A \delta^{3/2} -\delta + \delta^{1 + p / (4n)} \cdot \int_{\RR^n} |x| d \mu(x). $$
By our assumptions on the measure $\mu$, we know that $\int |x| d \mu(x)< \infty$. For a small, positive $\delta$, the leading term in $\cI_{\mu,p}(\psi_{\delta})$ is $-\delta$. Consequently, $\cI_{\mu,p}(\psi_{\delta}) < 0$
for a sufficiently small $\delta > 0$.
\end{proof}

In order to prove Proposition \ref{prop_340}, we select a minimizing sequence $$ \{ \psi_{\ell} \}_{\ell=1,2,\ldots, \infty}. $$
In other words, for any $\ell \geq 1$ the function $\psi_{\ell}: \RR^n \rightarrow \RR \cup \{ + \infty \}$ is a $\mu$-integrable, proper, convex function with $\psi_{\ell}(0) < 0$
and
$$ \cI_{\mu,p}(\psi_{\ell}) \stackrel{\ell \rightarrow \infty} \longrightarrow \inf_{\psi} \cI_{\mu, p}(\psi) $$
where the infimum runs over all $\mu$-integrable, proper, convex functions $\psi: \RR^n \rightarrow \RR \cup \{+\infty\}$ with $\psi(0) < 0$.
Thanks to Lemma \ref{lem_351}, we may select the sequence $\{ \psi_{\ell} \}$ so that
\begin{equation}
\sup_{\ell \geq 1} \cI_{\mu, p}(\psi_{\ell}) < 0.
\label{eq_351}
\end{equation}
Moreover, we know  that $\cI_{\mu, p}(\psi_{\ell})$ remains intact when we add a linear functional to $\psi_{\ell}$.
Arguing as in the proof of Proposition \ref{prop_515}, we may add appropriate linear functionals to $\psi_{\ell}$ and assume that
\begin{equation} \inf_{x \in \RR^n} \psi_{\ell}(x) = \psi_{\ell}(0)
\qquad \qquad \text{for} \ \ell \geq 1. \label{eq_842} \end{equation}

\begin{lemma} We have that $\sup_{\ell} \psi_{\ell}(0) < 0$ and $\inf_{\ell} \psi_{\ell}(0) > -\infty$. \label{lem_523}
\end{lemma}

\begin{proof} By (\ref{eq_842}), for any $\ell \geq 1$,
$$ \psi_{\ell}(0) = \inf_{x \in \RR^n} \psi_{\ell}(x) \leq \int_{\RR^n} \psi_{\ell} d \mu \leq \cI_{\mu,p}(\psi_{\ell}). $$
Inequality (\ref{eq_351}) thus implies that $\sup_{\ell} \psi_{\ell}(0) < 0$. Moreover,
it follows from (\ref{eq_351}) that $\int \psi_{\ell} d \mu < 0$ for all $\ell$. From (\ref{eq_351}), (\ref{eq_842}) and Lemma \ref{lem_406},
$$ \psi_{\ell}(0) + \tilde{c} ( \psi_{\ell}(0) )^2 \leq \cI_{\mu, p}(\psi_{\ell}) < 0 \qquad \qquad (\ell \geq 1). $$
Hence $\inf_{\ell} \psi_{\ell}(0) \geq - 1 / \tilde{c} > -\infty$.
\end{proof}

Write $K \subseteq \RR^n$ for the interior of $\conv(\supp(\mu))$.
Then $K$ is an open, convex set containing the origin.
Lemma 16 in \cite{CK} states that for any non-negative, $\mu$-integrable, convex function $f: \RR^n \rightarrow \RR \cup \{+ \infty \}$
and any point $x \in K$,
\begin{equation}
f(x) \leq C_{\mu}(x) \int_{\RR^n} f d \mu, \label{eq_520}
\end{equation}
where $C_{\mu}(x) > 0$ depends solely on $x$ and $\mu$.

\begin{lemma} There exists a sequence of integers $\{ \ell_j \}_{j=1,2,\ldots}$ such that
$\psi_{\ell_j}$ converges pointwise in $K$ to a certain convex function $\psi: K \rightarrow \RR$. \label{lem_532}
\end{lemma}

\begin{proof} Fix a point $x_0 \in K$. We claim that
\begin{equation}
\sup_{\ell \geq 1} |\psi_{\ell}(x_0)| < +\infty. \label{eq_530}
\end{equation}
Indeed, the fact that
the sequence $\{ \psi_{\ell}(x_0) \}_{\ell=1,2,\ldots}$ is bounded from
below follows from (\ref{eq_842}) and Lemma \ref{lem_523}.
In order to show that this sequence is bounded from above, we denote
\begin{equation} \beta = -\inf \left \{ \psi_{\ell}(x) \, ; \, x \in \RR^n, \ell \geq 1 \right \} =
-\inf \left \{ \psi_{\ell}(0) \, ; \, \ell \geq 1 \right \}
\label{eq_816} \end{equation}
which is a finite, positive number thanks to Lemma \ref{lem_523}.
Apply (\ref{eq_520}) for
the non-negative, $\mu$-integrable, convex function $f_{\ell} = \psi_{\ell} + \beta$, and obtain
\begin{align*}  f_{\ell}(x_0) & \leq C_{\mu}(x_0) \int_{\RR^n} f_{\ell}(x) d \mu(x) = C_{\mu}(x_0) \int_{\RR^n} ( \psi_{\ell} + \beta ) d\mu \\ & \leq C_{\mu}(x_0) \left( \beta + \cI_{\mu, p}(\psi_{\ell}) \right)
\leq C_{\mu}(x_0) \beta, \end{align*}
where we used (\ref{eq_351}) in the last passage.
This shows that $\sup_{\ell} f_{\ell}(x_0) < \infty$, and
consequently $\sup_{\ell} \psi_{\ell}(x_0) < \infty$. The proof of (\ref{eq_530}) is complete. We may now invoke Theorem 10.9 from Rockafellar \cite{roc},
thanks to (\ref{eq_530}), and conclude that there exists a subsequence $\{ \psi_{\ell_j} \}$
satisfying the conclusion of the lemma.
\end{proof}

\begin{proof}[Proof of Proposition \ref{prop_340}]
We will use the convergent subsequence $\{ \psi_{\ell_j} \}$ from Lemma \ref{lem_532}.
The  function $\psi = \lim_{j} \psi_{\ell_j}$ is finite and convex in the open, convex set $K$. Moreover,  $\psi(0) \in (-\infty, 0)$ as follows
from Lemma \ref{lem_523}. Since $\psi_{\ell}(x) \geq \psi_{\ell}(0)$ for any $x \in \RR^n$ and $\ell \geq 1$, also
\begin{equation}
\psi(0) = \inf_{x \in K} \psi(x) \in (-\infty, 0). \label{eq_803}
\end{equation}
 The function $\psi$ is currently defined only in the set $K$.
In order to have a globally defined function in $\RR^n$, we set $\psi(x) = +\infty$ for $x \in \RR^n \setminus \overline{K}$.
For $x \in \partial K$, define
\begin{equation} \psi(x) = \lim_{t \rightarrow 1^-} \psi(t x). \label{eq_804} \end{equation}
Since $\psi$ is convex in $K$, it follows from (\ref{eq_803}) that the function $t \mapsto \psi(t x)$ is non-decreasing in $t \in (0,1)$,
hence the limit in (\ref{eq_804}) is well-defined. Moreover, the function $\psi: \RR^n \rightarrow \RR \cup \{ + \infty \}$ is a proper, convex function,
since on $\overline{K}$ we have $\psi = \sup_{t \in (0,1)} f_t$ where $f_t(x) = \psi(t x)$ is finite, convex and continuous on $\overline{K}$.
The measure $\mu$ is supported in the closure $\overline{K}$.
From the pointwise convergence in $K$, it follows that $\psi_{\ell_j}(t x) \longrightarrow \psi(t x)$ for any $0 < t < 1$ and $x \in \overline{K}$.
We claim that by Fatou's lemma, for any $0 < t < 1$,
\begin{equation} \int_{\overline{K}} \psi(t x) d \mu(x) \leq \liminf_{j \rightarrow \infty} \int_{\overline{K}} \psi_{\ell_j}(t x) d \mu(x)
\leq \liminf_{j \rightarrow \infty} \int_{\overline{K}} \psi_{\ell_j}(x) d \mu(x). \label{eq_809}
\end{equation}
Indeed, the use of Fatou's lemma is legitimate according to (\ref{eq_842}) and Lemma \ref{lem_523}, because  $\inf_{x, \ell} \psi_{\ell}(x) > -\infty$.
The relation (\ref{eq_842}) also implies that $\psi_{\ell}(tx) \leq \psi_{\ell}(x)$ for any $x \in \overline{K}, \ell \geq 1$ and $0 < t < 1$, completing the justification
of (\ref{eq_809}). Next, we use the fact that $\psi(t x) \nearrow \psi(x)$ as $t \rightarrow 1^-$ for any $x \in \overline{K}$.
Since $\psi$ is bounded from below, we may use the monotone convergence theorem, and upgrade (\ref{eq_809}) to the bound
\begin{equation}
\int_{\RR^n} \psi d \mu = \int_{\overline{K}} \psi d \mu
= \lim_{t \rightarrow 1^-} \int_{\overline{K}} \psi(t x) d \mu(x)
\leq \liminf_{j \rightarrow \infty} \int_{\overline{K}} \psi_{\ell_j} d \mu = \liminf_{j \rightarrow \infty} \int_{\RR^n} \psi_{\ell_j} d \mu.
\label{eq_813}
\end{equation}
Recall from (\ref{eq_351}) that $\sup_j \int \psi_{\ell_j} d \mu < 0$.
It follows from (\ref{eq_803}) and (\ref{eq_813}) that $\psi$ is a $\mu$-integrable, proper, convex
function with $\psi(0) < 0$.
 All that remains is to prove that
\begin{equation}
\cI_{\mu, p}(\psi) \leq \liminf_{j \rightarrow \infty} \cI_{\mu, p}(\psi_{\ell_j}).
\label{eq_1043}
\end{equation}
The convex function $\psi$ satisfies $K \subseteq \dom(\psi) \subseteq \overline{K}$,
and $\psi_{\ell_j} \longrightarrow \psi$ pointwise in $K$ as $j \rightarrow \infty$. From
Lemma \ref{lem_1402},
\begin{equation} \cI_p(\psi) \leq \liminf_{j \rightarrow \infty}  \cI_p(\psi_{\ell_j})
\qquad \text{and hence} \qquad \cI_p^2(\psi) \leq \liminf_{j \rightarrow \infty}  \cI_p^2(\psi_{\ell_j}).
\label{eq_1038}
\end{equation}
Now (\ref{eq_1043}) follows from (\ref{eq_813}), (\ref{eq_1038}) and the definition of $\cI_{\mu, p}$.
\end{proof}

From the proof of Proposition \ref{prop_340} we see that the minimizer $\psi$ may be selected so that $\psi(x) =+\infty$
for any $x \in \RR^n \setminus \overline{K}$.
Theorem \ref{cor_515} now follows from Proposition \ref{prop_515},
 Proposition \ref{prop_340} and Lemma \ref{lem_351}.

\section{$q$-moment measures}
\setcounter{equation}{0}

Let $q > 0$ and let $\vphi: \RR^n \rightarrow \RR$ be a positive, convex function such that $Z_\vphi := \int_{\RR^n} \vphi^{-(n+q)} < \infty$.
The function $\vphi$ is differentiable almost everywhere in $\RR^n$ because  it is convex.
We define the {\it $q$-moment measure of $\vphi$} to be the push-forward
of the probability measure on $\RR^n$ with density  $Z_{\vphi}^{-1} / \vphi^{n+q}$
under the measurable map
$x \mapsto \nabla \vphi(x)$.
In other words, a Borel probability measure $\mu$  on $\RR^n$ is the $q$-moment measure of $\vphi$ if for any bounded, continuous function
$b: \RR^n \rightarrow \RR$,
\begin{equation}
\int_{\RR^n} b(y) d \mu(y) = \int_{\RR^n} \frac{b(\nabla \vphi(x))}{\vphi^{n+q}(x)} \frac{dx}{Z_{\vphi}}.
\label{eq_406} \end{equation}
The moment measure of $\vphi$ is a well-defined probability measure on $\RR^n$, whenever $\vphi$ is a positive,
convex function on $\RR^n$ such that $\vphi^{-(n+q)}$ is integrable.

\begin{lemma} Let $q > 0$ and let $\vphi: \RR^n \rightarrow \RR$ be a positive, convex function. Then the function $\vphi^{-(n+q)}$ is integrable
if and only if $ \lim_{|x| \rightarrow \infty} \vphi(x) = +\infty$. Moreover, in this case there exist $\alpha, \beta > 0$ such that $\vphi(x) \geq \alpha + \beta |x|$
for all $x \in \RR^n$.
\label{lem_457}
\end{lemma}

\begin{proof} Assume that $\vphi^{-(n+q)}$ is integrable. Then for any $R > 0$, the open convex set $\{ x \in \RR^n \, ; \, \vphi(x) < R \}$ has a finite volume
and hence it is bounded. Therefore $\lim_{|x| \rightarrow \infty} \vphi(x) = +\infty$. Conversely, assume that $\vphi(x)$ tends to infinity as $|x| \rightarrow \infty$.
Then there exists $R > 0$ such that $\vphi(x) \geq \vphi(0) + 1$ whenever $|x| \geq R$. By convexity, for any $|x| > R$,
$$ \vphi(0) + 1 \leq \vphi \left( \frac{R}{|x|} x \right) \leq \left(1 - \frac{R}{|x|} \right) \vphi(0) + \frac{R}{|x|} \vphi(x). $$
Therefore $\vphi(x) \geq \vphi(0) + |x| / R$ for all $|x| > R$. By continuity, $c = \min_{|x| \leq R} \vphi(x)$ is positive. Hence
 $\vphi(x) \geq c/2 + \min \{ 1/R, c / (2 R) \} \cdot |x|$ for all
$x \in \RR^n$, and $\vphi^{-(n+q)}$ is integrable.
\end{proof}

Lemma \ref{lem_457} demonstrates that if $\vphi^{-(n+q)}$ is integrable for some $q > 0$, then it is integrable for all $q > 0$.
The moment measures from \cite{CK} correspond in a sense to the case $q = \infty$, since in \cite{CK} we push forward the measure
on $\RR^n$ with density $\exp(-\vphi)$ via the map $x \mapsto \nabla \vphi(x)$.
For a convex function $\vphi: \RR^n \rightarrow \RR$ and for $\lambda > 0$  we say that $$ (\lambda \times  \vphi)(x) = \lambda \vphi(x / \lambda) \qquad \qquad \qquad (x \in \RR^n) $$
is the $\lambda$-dilation of $\vphi$. Note that the $q$-moment measure of $\vphi$ is exactly the same as the $q$-moment measure of its dilation $\lambda \times \vphi$,
assuming that one of these $q$-moment measures exists. It is also clear that replacing $\vphi(x)$ by its translation $\vphi(x - x_0)$, for some $x_0 \in \RR^n$,
does not have any effect on the resulting  $q$-moment measure.

\begin{theorem} Let $q > 1$ and let $\mu$ be a compactly-supported Borel probability measure on $\RR^n$ whose barycenter lies at the origin.
Assume that the origin is in  the interior of $\conv(\supp(\mu))$.

\medskip Then there exists a positive, convex function $\vphi: \RR^n \rightarrow \RR$ whose $q$-moment measure is $\mu$.
This convex function $\vphi$ is uniquely determined up to translation and dilation.
\label{thm2}
\end{theorem}

Theorem \ref{thm2} is a variant for $q$-moment measures of a result proven in \cite{CK} in the case of moment measures.
The case where $\mu$ is not compactly-supported will not be discussed in this paper, although we expect
that similarly to \cite{CK}, essential-continuity will play a role in the analysis of this case.
We also restrict our attention to the case $q > 1$.
The necessity of the barycenter condition in Theorem \ref{thm2} follows from:

\begin{proposition} Let $q > 1$ and let $\mu$ be a compactly-supported Borel probability measure on $\RR^n$.
Assume that $\mu$ is the $q$-moment measure of a positive, convex function $\vphi: \RR^n \rightarrow \RR$.
Then the barycenter of $\mu$ lies at the origin, which belongs to the interior of $\conv(\supp(\mu))$.
\label{prop_1049}
\end{proposition}

\begin{proof} We may substitute $b(x) = x_i$ in
(\ref{eq_406}), since $b$ is bounded on $\supp(\mu)$. This shows that for $i=1,\ldots,n$,
$$  \int_{\RR^n} x_i d \mu(x) = \int_{\RR^n} \frac{\partial_i \vphi}{\vphi^{n+q}} =
-\frac{1}{n+q-1} \int_{\RR^n} \partial_i \left( \frac{1}{\vphi^{n+q-1}}  \right) = 0, $$
along the lines of \cite[Lemma 4]{CK}. Therefore the barycenter of $\mu$ lies at the origin.
Assume by contradiction that the origin is not in the interior of $\conv(\supp(\mu))$.
Since the barycenter of $\mu$ lies at the origin, necessarily $\mu$ is supported in a hyperplane of the form
$H = \theta^{\perp}$ for some $\theta \in S^{n-1}$. Since $\mu$ is the $q$-moment measure of $\vphi$,
we see that \begin{equation} \label{eq_1046_} \partial_{\theta} \vphi(x) = \langle \nabla \vphi(x), \theta \rangle = 0 \qquad \qquad \text{for almost all} \ x \in \RR^n. \end{equation}
The function $\vphi$ is locally-Lipschitz in $\RR^n$, being a finite, convex function. The relation (\ref{eq_1046_}) shows that
$\vphi$ is constant on almost any line parallel to $\theta$, contradicting the integrability of $\vphi^{-(n+q)}$.
\end{proof}

The proof of Theorem \ref{thm2} occupies most of the remainder of this section.
Begin the proof  with the following:

\begin{lemma} Let $q > 1$ and let $\vphi: \RR^n \rightarrow \RR$ be a positive, convex function with $\int_{\RR^n} \vphi^{-(n+q)} < \infty$.
Write $\mu$ for the $q$-moment measure of $\vphi$, and assume that $\mu$ is compactly-supported. Set $\psi = \vphi^*$. Then,
$$ \int_{\RR^n} |\psi| d \mu < \infty. $$ \label{lem_947}
\end{lemma}

\begin{proof} It follows from the definition of the Legendre transform
that for any point $x \in \RR^n$ in which $\vphi$ is differentiable, $$ \langle x, \nabla \vphi(x) \rangle =   \psi(\nabla \vphi(x)) + \vphi(x). $$
For almost any $x \in \RR^n$ we have that $\nabla \vphi(x) \in \supp(\mu)$.
Since $\mu$ is compactly-supported, then $|\nabla \vphi(x)|$ is an $L^{\infty}$-function in $\RR^n$. Consequently,
$$ \int_{\RR^n} \vphi^{-(n+q)} \int_{\RR^n} |\psi| d \mu = \int_{\RR^n}
\frac{|\psi(\nabla \vphi(x))|}{\vphi^{n+q}(x)} dx \leq
 \int_{\RR^n}
\frac{|\langle x, \nabla \vphi(x) \rangle| + \vphi(x)}{\vphi^{n+q}(x)} dx< \infty,
$$
by Lemma \ref{lem_457}, since $q > 1$. This completes the proof.
\end{proof}

\begin{lemma} Let $A,p > 0$ and let $\mu$ be as in Theorem \ref{thm2}.
Let $\psi: \RR^n \rightarrow \RR \cup \{ + \infty \}$  be a $\mu$-integrable, proper, convex function
such that $\dom(\psi)$ is bounded.
For $t \in \RR$ denote $\psi_t = \psi + t$ and $\vphi_t = \psi_t^*$. Then
for any $t < -\psi(0)$, the function $\vphi_t: \RR^n \rightarrow \RR$ is a positive, convex function
with $\int_{\RR^n} \vphi_t^{-(n+p)} \in (0, \infty)$. Moreover, there exists $t < -\psi(0)$ with
$$ \int_{\RR^n} \vphi_t^{-(n+p)}(x) dx = A. $$
\label{lem_242}
\end{lemma}

\begin{proof}
The set $\dom(\psi)$ is assumed to be bounded. Set $L = 1 + \sup_{x \in \dom(\psi)} |x| < \infty$.
Denoting $\vphi = \psi^*$,  we learn from Corollary 13.3.3 in Rockafellar \cite{roc}
that the convex function $\vphi: \RR^n \rightarrow \RR$ is an $L$-Lipschitz function.
 Lemma \ref{lem_1340} implies that $\psi$ is finite in an open neighborhood of the origin.
Fix $t < -\psi(0)$. By the continuity of $\psi$ near the origin, there exists $\eps_t > 0$, depending on $\psi$ and $t$, such that
$$ \psi_t(x) < -\eps_t \qquad \qquad \text{when} \ |x| < \eps_t. $$
Hence, for any $y \in \RR^n$ and $t < -\psi(0)$,
\begin{equation}  \vphi_t(y) = \sup_{x \in \dom(\psi_t)} \left[ \langle x, y \rangle - \psi_t(x) \right] \geq \sup_{|x| < \eps_t} \left[ \langle x, y \rangle + \eps_t \right]
= \eps_t + \eps_t |y|. \label{eq_734_} \end{equation}
Set $t_0 = -\psi(0)$, and for $t \in (-\infty, t_0)$ define
\begin{equation} I(t) = \int_{\RR^n} \frac{dx}{(\vphi_t(x))^{n+p}}   = \int_{\RR^n} \frac{dx}{(\vphi(x) - t)^{n+p}}. \label{eq_1049} \end{equation}
It follows from (\ref{eq_734_}) that the function $\vphi_t^{-(n+p)}$ is integrable on $\RR^n$.
The positive function $\vphi: \RR^n \rightarrow \RR$ is $L$-Lipschitz, hence the integral of $\vphi_t^{-(n+p)}$ is positive.
The function $I$ is clearly non-decreasing in $t \in (-\infty, t_0)$, and by the monotone convergence theorem,
$I$ is continuous in $(-\infty, t_0)$. In order to conclude  the lemma by the mean value theorem, it suffices to prove that
$$ \lim_{t \rightarrow -\infty} I(t) = 0, \quad \lim_{t \rightarrow t_0^-} I(t) = +\infty.
$$
The fact that $I(t) \rightarrow 0$ as $t \rightarrow -\infty$ is evident from (\ref{eq_1049}) and the monotone convergence theorem.
It remains  to show that $I(t) \rightarrow +\infty$ as $t \rightarrow t_0^-$.
With any $t < t_0$ we associate a point $x_0(t) \in \RR^n$ that satisfies
$$ \vphi(x_0(t)) < \frac{t_0 - t}{2} + \inf_{x \in \RR^n} \vphi(x) =   \frac{t_0 - t}{2} - \psi(0) = \frac{t_0 - t}{2} + t_0. $$
For any $t < t_0$, denoting $r = (t_0 - t) / (2L)$, we see that $\vphi(x) \leq \vphi(x_0(t)) + (t_0 - t) / 2 $ for any $x$ in the ball $B(x_0(t), r)$. Therefore, for any $t < t_0$,
$$ I(t) =\int_{\RR^n} \frac{dx}{(\vphi(x) - t)^{n+p}} \geq \int_{B(x_0(t), r)} \frac{dx}{(\vphi(x) - t)^{n+p}} \geq
\frac{\kappa_n r^n}{(2t_0 - 2t)^{n+p}} = \frac{\kappa_n 2^{-2n-p} L^{-n}}{(t_0 - t)^p}
$$
where $\kappa_n = \Vol_n(B(0,1))$ is the volume of the Euclidean unit ball. Since $p > 0$,
$$ \lim_{t \rightarrow t_0^-} I(t) \geq \lim_{t \rightarrow t_0^-} \frac{\kappa_n 2^{-2n-p} L^{-n}}{(t_0 - t)^p} = +\infty $$
and the lemma is proven.
\end{proof}

\begin{lemma} Let $q > 1$ and let $\mu$ be as in Theorem \ref{thm2}.
Let $\psi: \RR^n \rightarrow \RR \cup \{ + \infty \}$  be the $\mu$-integrable, proper, convex function whose existence is guaranteed
by Theorem \ref{cor_515} with $p = q - 1$.

\medskip Denote $\vphi = \psi^*$. Then
$\vphi: \RR^n \rightarrow \RR$ is a positive function
and the probability measure $\nu$ on $\RR^n$ with density $Z_{\vphi}^{-1} / \vphi^{n+q}$
is well-defined. Moreover, for any function $\psi_1: \RR^n \rightarrow \RR \cup \{ + \infty \}$
of the form $\psi_1 = \psi + b$, with  $b: \RR^n \rightarrow \RR$ being a bounded function, we have
\begin{equation}  \int_{\RR^n} \psi d \mu + \int_{\RR^n} \psi^* d \nu \leq \int_{\RR^n} \psi_1 d \mu + \int_{\RR^n} \psi_1^* d \nu. \label{eq_734}
\end{equation}
\label{lemma_936}
\end{lemma}

\begin{proof} Write $\overline{K}$ for the closure of $\conv(\supp(\mu))$, a compact set in $\RR^n$.  Theorem \ref{cor_515} states that $\psi(0) < 0$ and that $\dom(\psi) \subseteq \overline{K}$. Therefore, by Lemma \ref{lem_242}, the function $\vphi: \RR^n \rightarrow \RR$ is a positive, convex function with
\begin{equation} \int_{\RR^n} \vphi^{-(n+p)} \in (0, +\infty). \label{eq_451_} \end{equation}
It thus follows from Lemma \ref{lem_457} that the probability measure $\nu$ is well-defined.
The function $\psi_1^{**}$ is proper, convex, and it satisfies $\psi - C \leq \psi_{1}^{**} \leq \psi_1 \leq \psi + C$ for some $C > 0$.
It suffices to prove (\ref{eq_734}) under the additional assumption that $\psi_1$ is proper and convex: Otherwise, replace $\psi_1$ with the smaller $\psi_1^{**}$,
and observe that the right-hand side of (\ref{eq_734}) cannot increase under such a replacement.

\medskip Hence we may assume that $\psi_1$ is a $\mu$-integrable, proper, convex function. Moreover, the convex set $\dom(\psi_1) = \dom(\psi)$ is bounded
according to Theorem \ref{cor_515}.
The right hand-side of (\ref{eq_734}) is not altered if we add a constant to the function $\psi_1$, since $\mu$ and $\nu$ are probability measures.
By adding an appropriate constant to $\psi_1$ and by using Lemma \ref{lem_242} and (\ref{eq_451_}), we may assume that the convex function $\psi_1$ satisfies that $\psi_1(0) < 0$ and
\begin{equation} \int_{\RR^n} \frac{dx}{\vphi_1^{n+p}(x)} = \int_{\RR^n} \frac{dx}{\vphi^{n+p}(x)} \label{eq_156} \end{equation}
where $\vphi_1 = \psi_1^*: \RR^n \rightarrow \RR$ is a positive function. Since $\psi_1(0) < 0$, by Theorem \ref{cor_515},
\begin{equation} \int_{\RR^n} \psi d \mu + \left( \int_{\RR^n} \frac{1}{\vphi^{n+p}} \right)^{-2/p}
\leq \int_{\RR^n} \psi_1 d \mu + \left( \int_{\RR^n} \frac{1}{\vphi_1^{n+p}} \right)^{-2/p}. \label{eq_1004} \end{equation}
 From (\ref{eq_156}) and (\ref{eq_1004}),
\begin{equation}
\int_{\RR^n} \psi d \mu \leq \int_{\RR^n} \psi_1 d \mu.
\label{eq_246} \end{equation}
Note  the elementary inequality
$$  \frac{n+p}{t^{n+p+1}} (t - s) \leq \frac{1}{s^{n+p}} - \frac{1}{t^{n+p}}
\qquad \qquad \qquad (s,t > 0) $$
which follows from the convexity of the function $t \mapsto t^{-(n+p)}$ on $(0, \infty)$. The latter inequality implies that
\begin{equation}  \int_{\RR^n} \left( \vphi - \vphi_1 \right) \frac{n+p}{\vphi^{n+p+1}} \leq \int_{\RR^n} \left[ \frac{1}{\vphi_1^{n+p}} -  \frac{1}{\vphi^{n+p}} \right] = 0 \label{eq_451}
\end{equation}
where we used (\ref{eq_156}) in the last passage. Since $\vphi_1 - \vphi$ is a bounded function, all integrals in
(\ref{eq_451}) converge.
From (\ref{eq_451}) and the definition of the measure $\nu$,
\begin{equation}
\int_{\RR^n} \vphi d \nu \leq \int_{\RR^n} \vphi_1 d \nu. \label{eq_245}
\end{equation}
The desired inequality (\ref{eq_734}) follows from (\ref{eq_246}) and (\ref{eq_245}).
\end{proof}

\begin{proof}[Proof of the existence part in Theorem \ref{thm2}]
Lemma \ref{lemma_936} is the variational problem associated with
 {\it optimal transportation}, see Brenier \cite{brenier}
and Gangbo and McCann \cite{gangboo_mccann}.
Let $\psi, \vphi = \psi^*$ and $\nu$ be as in Lemma \ref{lemma_936}.
Then $\vphi: \RR^n \rightarrow \RR$ is a positive, convex function on $\RR^n$.
A standard argument from \cite{brenier, gangboo_mccann}
leads
us from (\ref{eq_734}) to the conclusion
that $\nabla \vphi$ pushes forward the measure $\nu$ to the measure $\mu$.

\medskip Let us provide some details. The idea of this standard argument is to apply (\ref{eq_734}) with the function
$\psi_1 = \psi + \eps b$, where $\eps > 0$ is a small number and $b: \RR^n \rightarrow \RR$ is a bounded, continuous function.
Denoting $\psi_{\eps} = \psi + \eps b$ for $0 \leq \eps < 1$ and $\vphi_{\eps} = \psi_{\eps}^*$, one verifies that
$$ \left. \frac{d \vphi_\eps(x)}{d\eps} \right|_{\eps = 0} = -b(\nabla \vphi(x)) $$
at any point $x \in \RR^n$ in which $\vphi$ is differentiable (see, e.g., Berman and Berndtsson \cite[Lemma 2.7]{BB} for a short proof).
Consequently, by the bounded convergence theorem,
\begin{equation}  \left. \frac{d}{d \eps} \left( \int_{\RR^n} \psi_\eps d \mu + \int_{\RR^n} \vphi_\eps d \nu \right) \right|_{\eps =0 } = \int_{\RR^n} b(x) d \mu(x) - \int_{\RR^n} b(\nabla \vphi(x)) d \nu(x). \label{eq_1405} \end{equation}
However, the expression in (\ref{eq_1405}) must vanish according to (\ref{eq_734}).
Recalling that the density of $\nu$ is proportional to $\vphi^{-(n+q)}$, we conclude
that
(\ref{eq_406})
is valid for any bounded, continuous function $b$. Therefore $\mu$ is the $q$-moment measure of $\vphi$.
\end{proof}

Our next inequality is analogous to Theorem 8 from \cite{CK},
and may be viewed as an ``above tangent'' version of the Borell-Brascamp-Lieb inequality.

\begin{proposition}
Let $q > 1$ and let $\mu$ be as in Theorem \ref{thm2}. Suppose that  $\vphi_0: \RR^n \rightarrow (0, \infty)$
is a convex function whose $q$-moment measure is $\mu$. Denote $p = q-1$ and $\psi_0 = \vphi_0^*$.
Then $\psi_0$ is $\mu$-integrable, and for any $\mu$-integrable, proper, convex function $\psi_1: \RR^n \rightarrow \RR \cup \{ + \infty \}$
with $\psi_1(0) < 0$, denoting $\vphi_1 = \psi_1^*$,
$$ \left( \int_{\RR^n} \frac{1}{\vphi_1^{n+p}} \right)^{-2/p} \geq \left( \int_{\RR^n} \frac{1}{\vphi_0^{n+p}} \right)^{-2/p} + \frac{2(n+p)
\int_{\RR^n} \vphi_0^{-(n+p+1)} }{p
\left( \int_{\RR^n} \vphi_0^{-(n+p)} \right)^{\frac{p+2}{p}} }
\int_{\RR^n} (\psi_0 - \psi_1) d \mu. $$
\label{prop_858}
\end{proposition}

We begin the proof of Proposition \ref{prop_858} with two reductions:

\begin{lemma} It suffices to prove Proposition \ref{prop_858} under the additional requirements
that $\dom(\psi_1) \subseteq \dom(\psi_0)$ and that $\psi_1 - \psi_0$ is bounded from below on $\dom(\psi_0)$.
\label{lem_1432_}
\end{lemma}

\begin{proof} It follows from Lemma \ref{lem_457} that $\psi_0(0) < 0$.
For $N > 0$ and $x \in \RR^n$ define $f_N(x) = \max \{ \psi_1(x), \psi_0(x) - N \}$.
The functions $\psi_0$ and $\psi_1$ are negative at zero, and hence
$f_N$ is a proper, convex function on $\RR^n$ with $f_N(0) < 0$ and $\dom(f_N) \subseteq \dom(\psi_0)$.
The function $\psi_0$ is $\mu$-integrable according to Lemma \ref{lem_947}.
The $\mu$-integrability of $\psi_0$ and $\psi_1$ implies that $f_N$ is $\mu$-integrable.
Assuming that Proposition \ref{prop_858} is proven under the additional requirement in the formulation of the lemma, we may assert that
\begin{equation}  \left( \int_{\RR^n} \frac{1}{(f_N^*)^{n+p}} \right)^{-2/p} \geq \left( \int_{\RR^n} \frac{1}{\vphi_0^{n+p}} \right)^{-2/p} + \frac{2(n+p)
\int_{\RR^n} \vphi_0^{-(n+p+1)} }{p
\left( \int_{\RR^n} \vphi_0^{-(n+p)} \right)^{\frac{p+2}{p}} }
\int_{\RR^n} (\psi_0 - f_N) d \mu. \label{eq_1607} \end{equation}
All that remains is to prove that
\begin{equation}
\int_{\RR^n} \psi_1 d \mu = \lim_{N \rightarrow \infty} \int_{\RR^n} f_N d \mu \label{eq_1454}
\end{equation}
and
\begin{equation}
\int_{\RR^n} \frac{1}{\vphi_1^{n+p}}  \leq \liminf_{N \rightarrow \infty} \int_{\RR^n} \frac{1}{(f_N^*)^{n+p}}. \label{eq_1507}
\end{equation}
Since $f_N \geq \psi_1$ then $f_N^* \leq \vphi_1$
and $(f_N^*)^{-(n+p)} \geq \vphi_1^{-(n+p)}$. Hence (\ref{eq_1507}) holds trivially.
Note that $f_N \searrow \psi_1$ as $N \rightarrow \infty$ pointwise in $\dom(\psi_0)$.
Since $\psi_0$ is $\mu$-integrable, the set $\dom(\psi_0)$ has a full $\mu$-measure.
Consequently, $f_N(x) \searrow \psi_1(x)$ as $N \rightarrow \infty$ for $\mu$-almost any $x \in \RR^n$.
The monotone convergence theorem implies (\ref{eq_1454}).
\end{proof}

\begin{lemma} It suffices to prove Proposition \ref{prop_858} under the additional requirement
that $\dom(\psi_1) = \dom(\psi_0)$ and that $\psi_1 - \psi_0$ is bounded on $\dom(\psi_0)$.
\label{lem_1432}
\end{lemma}

\begin{proof} According to Lemma \ref{lem_1432_}, we may assume that for some $C > 0$,
\begin{equation}  \psi_1(x) + C \geq \psi_0(x)  \qquad \qquad (x \in \RR^n). \label{eq_1601}
\end{equation}
It follows from (\ref{eq_1601}) that for any $N > 0$,
\begin{equation} \vphi_0 - N \leq \max \{ \vphi_1, \vphi_0 - N \} \leq \vphi_0 + C.
\label{eq_1442}
\end{equation}
For $N > 0$, let us  define \begin{equation} g_N = \left( \max \{ \vphi_1, \vphi_0 - N \} \right)^*. \label{eq_1450} \end{equation}
Since $\vphi_0$ is a proper, convex function, it follows from (\ref{eq_1442}) that
$g_N: \RR^n \rightarrow \RR \cup \{ + \infty \} $ is a proper, convex function as well.
It also follows from (\ref{eq_1442}) that $\dom(g_N) = \dom(\psi_0)$ and that $g_N - \psi_0$ is a bounded
function on $\dom(\psi_0)$. The $\mu$-integrability of $\psi_0$, proved in Lemma \ref{lem_947}, implies that $g_N$ is $\mu$-integrable.
We learn from (\ref{eq_1450}) that $g_N(0) \leq \psi_1(0) < 0$.
Assuming that Proposition \ref{prop_858} is proven under the additional requirement in the formulation of this lemma, we may assert that
(\ref{eq_1607}) holds true when $f_N$ is replaced by $g_N$.
All that remains to prove is that
\begin{equation}
\int_{\RR^n} \psi_1 d \mu \geq \limsup_{N \rightarrow \infty} \int_{\RR^n} g_N d \mu \label{eq_1454_}
\end{equation}
and
\begin{equation}
\int_{\RR^n} \frac{1}{\vphi_1^{n+p}}  \leq \liminf_{N \rightarrow \infty} \int_{\RR^n} \frac{1}{(g_N^*)^{n+p}}. \label{eq_1507_}
\end{equation}
Since $\psi_1 \geq g_N$ then (\ref{eq_1454_}) holds trivially.
Since $\dom(\vphi_0) = \RR^n$, it follows from (\ref{eq_1450}) that $$
g_N^* =  \max \{ \vphi_1, \vphi_0 - N \} \stackrel{N \rightarrow \infty}{\longrightarrow} \vphi_1 $$
pointwise in $\RR^n$. Now (\ref{eq_1507_}) follows from Fatou's lemma.
\end{proof}

\begin{proof}[Proof of Proposition \ref{prop_858}] The $\mu$-integrability of $\psi_0$ follows from Lemma \ref{lem_947},
while Lemma \ref{lem_457} implies that $\inf \vphi_0 > 0$.
According to Lemma \ref{lem_1432}, we may assume that $\dom(\psi_0) = \dom(\psi_1)$, and that
\begin{equation}  M = \sup_{\dom(\psi_0)} |\psi_1 - \psi_0| < \infty. \label{eq_1232} \end{equation}
Denote $f(x) = \psi_0(x) - \psi_1(x)$ for $x \in \dom(\psi_0)$ and $f(x) = +\infty$ for $x \not \in \dom(\psi_0)$. Set
 $\psi_t = (1 - t) \psi_0 + t \psi_1$ and $\vphi_t = \psi_t^*$.
 Thus $\dom(\psi_t) = \dom(\psi_0)$ while $\psi_t = \psi_0 - t f$ in the set $\dom(\psi_0)$.
 At any point  $x \in \RR^n$ in which $\vphi_0$ is differentiable, for any $0 \leq t \leq 1$,
\begin{equation} \vphi_t(x) = \psi_t^*(x) = \sup_{y \in \dom(\psi_0)} \left [ \langle x, y \rangle - \psi_0(y) + t f(y) \right]
\stackrel{``y = \nabla \vphi_0(x)"}{\geq}  \vphi_0(x) + t f( \nabla \vphi_0(x)). \label{eq_431}
\end{equation}
Denote $m = \inf \vphi_0$, which is a finite, positive number, thanks to the integrability of $\vphi_0^{-(n+q)}$ and to Lemma \ref{lem_457}.
By the Lagrange mean-value theorem from calculus,  for any $a,b,t \in \RR$ with $0 < t < m / (2 M), a \geq m$ and $|b| \leq M$,
\begin{equation} \frac{1}{t} \left[ \frac{1}{(a + t b)^{n+p}} - \frac{1}{a^{n+p}} \right] = -\frac{n+p}{\xi^{n+p+1}} b \leq -\frac{n+p}{a^{n+p+1}} b +
\frac{C_{n,p,m,M} }{a^{n+p+1}} \cdot t \label{eq_438} \end{equation}
for some $\xi$ between $a$ and $a + t b$, where $C_{n,p,m,M} > 0$ depends only on $n, p, m$ and $M$.
It follows from (\ref{eq_431}) and (\ref{eq_438}) that for any $t \in (0, m/(2M))$,
\begin{align} \label{eq_444}
\frac{1}{t} \int_{\RR^n} & \left[ \frac{1}{\vphi_t^{n+p}} - \frac{1}{\vphi_0^{n+p}} \right]  \leq
\frac{1}{t} \int_{\RR^n} \left[ \frac{1}{(\vphi_0(x) + t f( \nabla \vphi_0(x)) )^{n+p}}  -
\frac{1}{\vphi_0^{n+p}(x)} \right] dx \\ & \leq -(n+p) \int_{\RR^n} \frac{f \circ \nabla \vphi_0}{\vphi_0^{n+p+1}}
+  C t\int_{\RR^n} \frac{1}{\vphi_0^{n+p+1}} \stackrel{t \rightarrow 0^+}{\longrightarrow} -(n+p) \int_{\RR^n} \frac{f \circ \nabla \vphi_0}{\vphi_0^{n+p+1}},
\nonumber \end{align}
where $C = C_{n,p,m,M}$ and we used the facts that $\vphi_0^{-(n+p+1)}$ is integrable and that $f \circ \nabla \vphi_0$ is an $L^{\infty}$-function.
The relation (\ref{eq_1232}) implies that $|\vphi_0(x) - \vphi_1(x)| \leq M$ for all $x \in \RR^n$.
Hence $\dom(\vphi_0) = \dom(\vphi_1) = \RR^n$. Consequently, the function
$$ I(t) = \left( \int_{\RR^n} \frac{1}{\vphi_t^{n+p}} \right)^{-2/p} \qquad \qquad \qquad (0 \leq t \leq 1) $$
satisfies $I(0), I(1) \in [0, +\infty)$. By Lemma \ref{lem_1530}, the function $I$ is the square of
 a non-negative, convex funtion in the interval $[0,1]$. Therefore $I$  is a convex function.
Consequently,  the function $I$ is finite and upper semi-continuous in $[0,1]$, being a convex function in the interval $[0,1]$ which is finite at the endpoints of the interval.
The lower semi-continuity of $I$ at the origin follows from (\ref{eq_444}). Hence $I$ is continuous at the origin,
and by convexity,
\begin{align}  \nonumber I(1) - I(0) & \geq \liminf_{t \rightarrow 0^+} \frac{I(t) - I(0)}{t} \\ & \nonumber = -\frac{2}{p} \left( \int_{\RR^n} \frac{1}{\vphi_0^{n+p}} \right)^{-\frac{p+2}{p}}
\cdot \limsup_{t \rightarrow 0^+} \frac{1}{t} \int_{\RR^n} \left[ \frac{1}{\vphi_t^{n+p}} - \frac{1}{\vphi_0^{n+p}} \right]
\\ &\geq \frac{2(n+p)}{p} \left( \int_{\RR^n} \frac{1}{\vphi_0^{n+p}} \right)^{-\frac{p+2}{p}} \int_{\RR^n} \frac{f \circ \nabla \vphi_0}{\vphi_0^{n+p+1}},
\label{eq_1007} \end{align}
where we used (\ref{eq_444}) in the last passage. The proposition follows from (\ref{eq_1007}) and from the definition of $\mu$ as the $q$-moment measure of $\vphi_0$.
\end{proof}

The proof of Proposition \ref{prop_858}
looks rather different from the transportation proof of Theorem 8 in \cite{CK}.
The main difference is that above we apply the Borell-Brascamp-Lieb inequality
in the form of Lemma \ref{lem_1530}, while in \cite{CK} we essentially reprove the
Pr\'ekopa theorem.

\begin{proof}[Proof of the uniqueness part in Theorem \ref{thm2}]
Assume that $\vphi_0, \vphi_1: \RR^n \rightarrow (0, +\infty)$ are convex functions whose $q$-moment measure is $\mu$.
Our goal is to prove that there exist $\lambda > 0$ and $x_0 \in \RR^n$ such that
\begin{equation} \vphi_0(x) = \lambda \vphi_1( x_0 + x / \lambda ) \qquad \qquad \qquad \text{for} \ x \in \RR^n.
\label{eq_904}
\end{equation}
By Lemma \ref{lem_457}, the integrals $\int_{\RR^n} \vphi_i^{-(n+r)}$ converge for all $r > 0$ and $i =0,1$,
since $\vphi_0$ and $\vphi_1$ possess $q$-moment measures.
Replacing $\vphi_0(x)$ by its dilation $(\lambda \times \vphi_0)(x) = \lambda \vphi_0(x / \lambda)$, we may assume that
\begin{equation} \left( \int_{\RR^n} \frac{1}{\vphi_0^{n+p}} \right)^{-\frac{p+2}{p}} \int_{\RR^n} \frac{1}{\vphi_0^{n+p+1}} =
\left( \int_{\RR^n} \frac{1}{\vphi_1^{n+p}} \right)^{-\frac{p+2}{p}} \int_{\RR^n} \frac{1}{\vphi_1^{n+p+1}}. \label{eq_1053} \end{equation}
Indeed, replacing $\vphi_0$ by  $\lambda \times \vphi_0$ has the effect of multiplying the left-hand side of (\ref{eq_1053}) by $\lambda$,
hence we may select the appropriate dilation of $\vphi_0$ and assume that (\ref{eq_1053}) holds true. Denote $\psi_i = \vphi_i^*$ for $i=0,1$ and set
$$ \psi_{1/2} = (\psi_0 + \psi_1) / 2. $$
It follows from Lemma \ref{lem_457} that $\inf \vphi_i > 0$ for $i=0,1$. Therefore $\psi_i(0) = -\inf \vphi_i < 0$ for $i=0,1$
and consequently $\psi_{1/2}(0) < 0$. Denote $\vphi_{1/2} = \psi_{1/2}^*$.
Lemma \ref{lem_1530} implies that
\begin{equation}
\left( \int_{\RR^n} \frac{1}{\vphi_{1/2}^{n+p}} \right)^{-1/p} \leq \frac{1}{2} \left[
\left( \int_{\RR^n} \frac{1}{\vphi_0^{n+p}} \right)^{-1/p} + \left( \int_{\RR^n} \frac{1}{\vphi_1^{n+p}} \right)^{-1/p} \right]. \label{eq_1532} \end{equation}
According to Lemma \ref{lem_1530}(iii), when equality holds in (\ref{eq_1532}),  there exist $\lambda > 0$ and $x_0 \in \RR^n$ for which (\ref{eq_904}) holds true. All that remains to show is that equality holds in (\ref{eq_1532}).
The functions $\psi_0$ and $\psi_1$ are $\mu$-integrable, according to Lemma \ref{lem_947}.
Hence also $\psi_{1/2} = (\psi_0 + \psi_1)/2$ is $\mu$-integrable. Denote by $\alpha$ the quantity in (\ref{eq_1053}).
Applying
Proposition \ref{prop_858} for $\psi_0$ and $\psi_{1/2}$ we obtain
$$ \left( \int_{\RR^n} \frac{1}{\vphi_{1/2}^{n+p}} \right)^{-2/p} \geq \left( \int_{\RR^n} \frac{1}{\vphi_0^{n+p}} \right)^{-2/p} + \frac{2(n+p)}{p}
\alpha
\int_{\RR^n} (\psi_0 - \psi_{1/2}) d \mu. $$
Applying Proposition \ref{prop_858} for $\psi_1$ and $\psi_{1/2}$ we obtain
$$ \left( \int_{\RR^n} \frac{1}{\vphi_{1/2}^{n+p}} \right)^{-2/p} \geq \left( \int_{\RR^n} \frac{1}{\vphi_1^{n+p}} \right)^{-2/p} + \frac{2(n+p)}{p}
\alpha
\int_{\RR^n} (\psi_1 - \psi_{1/2}) d \mu. $$
Adding these two inequalities, and using $2 \psi_{1/2} = \psi_0 + \psi_1$, we have
\begin{equation} \left( \int_{\RR^n} \frac{1}{\vphi_{1/2}^{n+p}} \right)^{-2/p} \geq \frac{1}{2} \left[ \left( \int_{\RR^n} \frac{1}{\vphi_0^{n+p}} \right)^{-2/p} +
\left( \int_{\RR^n} \frac{1}{\vphi_1^{n+p}} \right)^{-2/p} \right]. \label{eq_930}
\end{equation}
From  (\ref{eq_930}) we deduce that equality holds in (\ref{eq_1532}), because  $\sqrt{(a^2 + b^2) / 2 }\geq (a+b)/2$ for all $a,b \geq 0$.
This completes the proof.\end{proof}

For a smooth function $f: \RR^n \rightarrow \RR$ we write $\nabla^2 f(x)$ for the Hessian matrix of $f$ at the point $x \in \RR^n$.
A smooth function $f: L \rightarrow \RR$ is {\it strongly-convex}, where $L \subseteq \RR^n$ is a convex, open set,
if $\nabla^2 f(x)$ is positive-definite for any $x \in L$.
Suppose that $L \subseteq \RR^n$ is a non-empty, open, bounded, convex set.
We are interested in smooth, convex solutions $\vphi: \RR^n \rightarrow (0, \infty)$ to the equation
with the constraint
\begin{equation}
\left \{ \begin{array}{lr} \det \nabla^2 \vphi = C / \vphi^{n+2} & \text{in} \ \RR^n \\
\nabla \vphi(\RR^n) = L \end{array} \right. \label{eq_1502}
\end{equation}
where $C > 0$ is a positive number. Here, of course, $\nabla \vphi(\RR^n) = \{ \nabla \vphi(x) \, ; \, x \in \RR^n \}$.
Thanks to the regularity theory for optimal transportation developed by Caffarelli \cite{caf1} and Urbas \cite{urbas},
Theorem \ref{thm2} admits the following corollary:

\begin{theorem} Let $L \subseteq \RR^n$ be a non-empty, open, bounded, convex set.
Then there exists a smooth, positive, convex function $\vphi: \RR^n \rightarrow \RR$ solving (\ref{eq_1502})
if and only if the barycenter of $L$ lies at the origin.
Moreover, this convex function $\vphi$ is uniquely determined up to translation and dilation. \label{thm3}
\end{theorem}

\begin{proof} Let $\mu$ be the uniform measure on $L$, normalized to be a probability measure.
Assume first  that the barycenter of $L$ lies at the origin. 
Then the origin belongs to the interior of $\supp(\mu)$.
Applying Theorem \ref{thm2} with $q = 2$ we obtain a positive convex function $\vphi: \RR^n \rightarrow \RR$
whose $q$-moment measure is $\mu$. That is, for any bounded, continuous function
$b: L \rightarrow \RR$,
\begin{equation}
\int_{L} b(y) d y = C_{L, \vphi} \int_{\RR^n} \frac{b(\nabla \vphi(x))}{\vphi^{n+2}(x)} dx,
\label{eq_1509} \end{equation}
where $C_{L, \vphi} = \Vol_n(L) / \int_{\RR^n} \vphi^{-(n+2)}$.
Caffarelli's regularity theory for optimal transportation  (see \cite{caf1} and the Appendix in  \cite{ADM})
implies  that $\vphi$ is $C^{\infty}$-smooth in $\RR^n$.
It follows from (\ref{eq_1509}) and from the change-of-variables formula that for any $x \in \RR^n$,
\begin{equation} \det \nabla^2 \vphi(x) = \frac{C_{L, \vphi}}{\vphi^{n+2}(x)}.
\label{eq_1120}
\end{equation}
In particular, the Hessian $\nabla^2 \vphi(x)$ is invertible and hence positive-definite for any $x \in \RR^n$.
Since $\vphi: \RR^n \rightarrow \RR$ is a smooth, strongly-convex function, the set $\nabla \vphi(\RR^n)$ is convex and open,
according to Theorem 26.5 in Rockafellar \cite{roc} or to Section 1.2 in Gromov \cite{gro}.
From (\ref{eq_1509}) we obtain that $\nabla \vphi(\RR^n) = L$,
thus $\vphi$ solves (\ref{eq_1502}).

 \medskip Moreover, we claim that the smooth, positive, convex solution $\vphi$ to (\ref{eq_1502}) is uniquely determined up to translation and dilation.
 Indeed, any such solution $\vphi$ is strongly-convex, and consequently $\nabla \vphi$ is a diffeomorphism between $\RR^n$
 and the convex, open set $\nabla \vphi(\RR^n) = L$. From (\ref{eq_1502}) and the change-of-variables formula we thus learn that
 $\mu$ is the $q$-moment measure of $\vphi$ with $q=2$. By Theorem \ref{thm2}, the function $\vphi$ is uniquely determined up to translation and dilation.

 \medskip  In order to prove the other direction of the theorem, assume that $\vphi$ is a smooth, positive, convex solution to (\ref{eq_1502}).
 As explained in the preceding paragraphs, $\mu$ is the $q$-moment measure of $\vphi$, with $q = 2$.
Proposition \ref{prop_1049} now shows that the barycenter of $\mu$ lies at the origin.
\end{proof}

\section{The affine hemisphere equations}
\setcounter{equation}{0}

In this section we review the partial differential equations for affinely-spherical hypersurfaces described by
Tzitz\'eica \cite{Tz1, Tz2}, Blaschke \cite{Bla} and Calabi \cite{Cal2}.
 Recall from Section \ref{sec_intro} the definition of the {\it affine normal line} $\ell_M(y)$
which is a line in $\RR^{n+1}$ passing through the point $y$ of the smooth, connected, locally strongly-convex hypersurface
$M \subset \RR^{n+1}$.
We use $y = (x,t) \in \RR^n \times \RR$ as coordinates in $\RR^{n+1}$.
For a set $L \subseteq \RR^n$
and a function $\psi: L \rightarrow \RR$ denote
$$ \graph_L(\psi) = \left \{ (x, \psi(x)) \, ; \, x \in L \right \} \subseteq \RR^{n} \times \RR = \RR^{n+1}. $$
The affine normal line $\ell_M(y)$  depends on the third order approximation to $M$ near $y$, as shown in the following lemma:

\begin{lemma}
Let $M  \subset \RR^{n+1}$ be a smooth, connected, locally strongly-convex hypersurface. Let $L \subseteq \RR^n$ be an open, convex set containing the origin.
Assume that $U \subseteq \RR^{n+1}$ is an open set such that
$$ M \cap U = \graph_L(\psi) $$
where $\psi: L \rightarrow \RR$ is
 a smooth, strongly-convex function
with $\psi(0) = 0, \nabla \psi(0) = 0$ and $\nabla^2 \psi(0) = \id$. Here, $\id$ is the identity matrix.

\medskip Then for $y_0 = (0,0) \in M$,
the line $\ell_{M}(y_0)$ is the line passing through the point $y_0$ in the direction of the  vector
\begin{equation}  \left( -\left( \nabla^2 \psi(0) \right)^{-1} \cdot \nabla (\log \det \nabla^2 \psi)(0), n+2 \right) \in \RR^n \times \RR = \RR^{n+1}. \label{eq_142}
\end{equation}
\label{lem_1151}
\end{lemma}

\begin{proof}
The vector $v = (0,1) \in \RR^n \times \RR$ is pointing to the convex side of $M$ at the point $y_0$.
The tangent space to $M$ at the point $y_0$ is $H = T_{y_0} M = \{ (x,0) \, ; \, x \in \RR^n \}$.
For a sufficiently small $t > 0$, the section $M_t = M \cap (H + t v)$ encloses
an $n$-dimensional convex body $\Omega_t \subset H + t v$ given by
$$ \Omega_{t} =  \left \{ (x,t) \in \RR^{n} \times \RR \, ; \,  \psi(x) \leq t \right \}. $$
Denote $a_{ijk} = \partial^{ijk} \psi(0) = \frac{\partial^3 \psi}{\partial x_i \partial x_j \partial x_k} (0) $.
By Taylor's theorem, for a sufficiently small $t > 0$,
\begin{align*}
 \Omega_{t} =
\left \{ (x,t) \in \RR^{n} \times \RR \, ; \,  \frac{|x|^2}{2} + \frac{1}{6} \sum_{i,j,k=1}^n a_{ijk} x_i x_j x_k + O(|x|^4) \leq t \right \}, \end{align*}
where $O(|x|^4)$ is an abbreviation for an expression that is bounded in absolute value by $C |x|^4$,
where $C$ depends only on $M$. By using to the spherical-coordinates representation of $\Omega_t$, we see that for a sufficiently small $t >0$,
$$ \frac{\Omega_{t/2}}{\sqrt{t}} = \left\{ \left(r \theta, \sqrt {t}/2 \right) \, ; \, \theta \in S^{n-1}, 0 \leq r \leq r_t(\theta) = 1 - \frac{\sum_{i,j,k=1}^n a_{ijk} \theta_i \theta_j \theta_k}{6} \sqrt{t} + O(t) \right \}, $$
where $t^{-1/2}  \cdot \Omega_{t/2} = \{ y / \sqrt{t} \, ; \, y \in \Omega_{t/2} \}$.
Consequently, the barycenter satisfies $\bary(\Omega_{t/2}) = (x_t, t/2)$ for
$$  x_t = \sqrt{t} \frac{n \int_{S^{n-1}} \theta \, r_t(\theta)^{n+1} d \theta }{(n+1) \int_{S^{n-1}} r_t(\theta)^{n} d \theta} = -t \cdot \frac{n}{6} \cdot \int_{S^{n-1}} \theta \left( \sum_{i,j,k=1}^n a_{ijk} \theta_i \theta_j \theta_k \right) d \sigma_{n-1}(\theta) + O(t^{3/2}), $$
where $\sigma_{n-1}$ is the uniform probability measure on $S^{n-1}$. Let $X = ( X_1,\ldots,X_n)$ be a standard Gaussian random vector in $\RR^n$,
and recall that $\EE X_i^2 = 1$ and $\EE X_i^4 = 3$ for all $i$. For any homogenous polynomial $p$ of degree $4$ in $n$ real variables, we know that $\EE p(X) = n (n+2) \int_{S^{n-1}} p(\theta) d \sigma_{n-1}(\theta)$.
Hence,
\begin{equation*}  \bary(\Omega_{t/2}) = \left(   -t \frac{n}{6 n (n+2) } \EE X \left[ \sum_{i,j,k=1}^n a_{ijk} X_i X_j X_k \right]  + O(t^{3/2}), t/2 \right).
 \end{equation*}
Consequently, the line $\ell_{M}(y_0)$ is in the direction of  the vector \begin{equation*}  \left( -\EE X \left[ \sum_{i,j,k=1}^n \partial^{ijk} \psi(0) X_i X_j X_k \right], 3 (n+2) \right) = \left( -3 \nabla (\Delta \psi)(0), 3 (n+2)  \right), \end{equation*}
 where $\Delta \psi = \sum_{i=1}^n \partial^{ii} \psi$.
Since $\nabla^2 \psi(0) = \id$, we see that $\nabla (\Delta \psi)(0) = \left( \nabla^2 \psi(0) \right)^{-1} \cdot \nabla (\log \det \nabla^2 \psi)(0)$, and the lemma is proven.
 \end{proof}

Suppose
that $V$ is a finite-dimensional linear space over $\RR$, and let $\psi: V \rightarrow \RR$ be a smooth, strongly-convex function.
In general it is impossible to identify a specific vector in $V$ as the gradient
of the function $\psi$ at the origin,
unless we introduce additional structure such as a scalar product.
Nevertheless, a simple and useful observation is that the vector
\begin{equation}
  \left( \nabla^2 \psi(0) \right)^{-1} \cdot \nabla \left( \log \det \nabla^2 \psi \right)(0)
\label{eq_409_}
   \end{equation}
is a well-defined vector in $V$. This means that for any scalar product that one may introduce in $V$,
we may compute the expression in (\ref{eq_409_}) relative to this scalar product, and the result will always
be the same vector in $V$.

\begin{lemma} Let $M \subset \RR^{n+1}$ be a hypersurface and let $L \subseteq \RR^n$ be a non-empty, open, convex set.  Suppose that
$\psi: \RR^n \rightarrow \RR \cup \{ +\infty \}$ is a proper,
convex function whose restriction to the set $L$ is finite, smooth and strongly convex. Denote
$\Lambda(x) = \log \det \nabla^2 \psi(x)$ for $x \in L$. Assume that
$$ M = \graph_L(\psi). $$
Let $x_0 \in L$ and denote $y_0 = (x_0, \psi(x_0)) \in M$. Then the affine normal line
$\ell_{M}(y_0) \subseteq \RR^{n+1}$ is the line passing through the point $y_0 \in \RR^{n+1}$ in the direction of the vector
\begin{equation} \left( -\left( \nabla^2 \psi \right)^{-1} \nabla \Lambda, n+2 -  \left \langle \left( \nabla^2 \psi \right)^{-1} \nabla \Lambda, \nabla \psi \right \rangle \right) \in \RR^n \times \RR =\RR^{n+1}, \label{eq_247}
\end{equation}
where all expressions are evaluated at the point $x_0$.
\label{prop_1242}
\end{lemma}

 \begin{proof} Translating, we may assume that $x_0 = 0$ and $\psi(0) = 0$. Consider first the case where also $\nabla \psi(0) = 0$.
In this case, the vector in (\ref{eq_247}) does not depend on the choice of the Euclidean structure in $\RR^n$,
hence we may switch to a Euclidean structure for which $\nabla^2 \psi(0) = \id$. Thus (\ref{eq_247}) follows
from Lemma \ref{lem_1151} in this case.
 In the case where $v := \nabla \psi(0)$ is a non-zero vector,
  we apply the linear map in $\RR^{n+1}$,
 $$ (x,t) \mapsto \left( x, t - \langle x, v \rangle \right). $$ This linear map transforms $M$ to the graph of the convex function $\psi_1(x) =\psi(x) - \langle x, v \rangle$,
 and it transforms the vector in (\ref{eq_247}) to the vector
$$ \left( -\left( \nabla^2 \psi_1(0) \right)^{-1} \cdot \nabla (\log \det \nabla^2 \psi_1)(0), n+2 \right) \in \RR^{n+1}.
$$
Since $\nabla \psi_1(0) = 0$, we have reduced matters to the case already proven.
 \end{proof}

\begin{remark}{\rm The affine normal lines considered in this paper are closely related to the {\it affine normal field} which is
discussed, e.g., by Nomizu and Sasaki \cite[Section II.3]{NS}.
The affine normal field is a certain map $\xi: M \rightarrow \RR^{n+1}$
that is well-defined whenever $M  \subseteq \RR^{n+1}$ is  a smooth, connected, locally strongly-convex hypersurface.
The relation between the affine normal field
and the affine normal line is simple: For any $y \in M$, the affine normal field $\xi_y$ is pointing
in the direction of the affine normal line $\ell_M(y)$. Indeed, using affine-invariance it suffices
to verify this in the case where $M = \graph_L(\psi)$. Example 3.3 in \cite[Section II.3]{NS} demonstrates
that when $M = \graph_L(\psi)$, for any $x \in L$ and $y = (x, \psi(x)) \in M$,
\begin{equation} \xi_y = \frac{(\det \nabla^2 \psi)^{1/(n+2)}}{n+2} \cdot \left( -(\nabla^2 \psi)^{-1} \nabla \Lambda , n+2 - \left \langle (\nabla^2 \psi)^{-1} \nabla \Lambda, \nabla \psi \right \rangle \right)  \in \RR^n \times \RR, \label{eq_924}
\end{equation}
where $\Lambda = \log \det \nabla^2 \psi$ and all expressions involving $\psi$ and $\Lambda$ are evaluated at the point $x$.
The vector in (\ref{eq_924}) is proportional to the vector described in Lemma \ref{prop_1242}, and hence $\xi_y$ is pointing in the
direction of the line $\ell_M(y)$.
}\end{remark}

\begin{proposition}
Let $M, L$ and $\psi$ be as in Lemma  \ref{prop_1242}. Denote $\vphi = \psi^*$ and $\Omega = \nabla \psi(L) = \{ \nabla \psi(x) \, ; \, x \in L \}$.
Then the following hold:
\begin{enumerate}
\item[(i)] The set $\Omega \subseteq \RR^n$ is open and the function $\vphi$ is smooth in $\Omega$.
\item[(ii)] The hypersurface $M$ is affinely-spherical with center at the origin if and only if there exists $C \in \RR \setminus \{ 0 \}$ such that
\begin{equation} \vphi^{n+2} \cdot \det \nabla^2 \vphi = C \qquad \qquad \text{in the entire set} \ \ \Omega. \label{eq_343} \end{equation}
\end{enumerate}
\label{prop_1147_}
\end{proposition}

\begin{proof}
The function $\psi$ is smooth and strongly-convex in the open, convex set $L$.
By strong-convexity, the smooth map $\nabla \psi: L \rightarrow \Omega$
is one-to-one (see, e.g., \cite[Theorem 26.5]{roc}).
Moreover, the differential of the smooth map $\nabla \psi: L \rightarrow \Omega$
is non-singular, and by the inverse function theorem from calculus,
the set $\Omega = \nabla \psi(L)$ is open and the map $\nabla \psi: L \rightarrow \Omega$ is a diffeomorphism.
According to \cite[Corollary 23.5.1]{roc}, the inverse of the map $\nabla \psi$ is the smooth map $\nabla \vphi: \Omega \rightarrow L$, and hence
\begin{equation}  \nabla^2 \vphi = (\nabla^2 \psi)^{-1} \circ \nabla \vphi. \label{eq_1025} \end{equation}
Thus (i) is proven. We move on to the proof of (ii). Assume first that $M$ is affinely-spherical with center at the origin. Then
for any $x \in L$, the vector in (\ref{eq_247}) is proportional to $(x, \psi(x))$. That is, for any $x \in L$,
\begin{equation}
 -\psi(x) \left( \nabla^2 \psi \right)^{-1} \nabla \left(\log \det \nabla^2 \psi \right) = \left[ n+2 -  \left \langle \left( \nabla^2 \psi \right)^{-1} \nabla \left(\log \det \nabla^2 \psi \right), \nabla \psi \right \rangle \right] x. \label{eq_1121} \end{equation}
By using  the shorter Einstein notation we may repharse (\ref{eq_1121}) as follows: for $x \in L$ and $i=1,\ldots,n$,
\begin{equation}  -\psi \psi^{ik}_{k} = \left( n + 2 - \psi^{jk}_{k} \psi_j \right) x^i. \label{eq_1121_}
\end{equation}
Let us briefly explain this standard notation. We denote $x = (x^1,\ldots,x^n) \in \RR^n, \nabla^2 \psi(x) = (\psi_{ij}(x))_{i,j=1,\ldots,n}$
and $(\nabla^2 \psi)^{-1}(x) = (\psi^{ij}(x))_{i,j=1,\ldots,n}$. We abbreviate $\psi_{ij}^{k} = \sum_{\ell=1}^n
\psi^{\ell k} \psi_{ij\ell}$ and $\psi^{ij}_k = \sum_{\ell, m=1}^n \psi^{i \ell} \psi^{j m} \psi_{\ell m k}$, where $\psi_{ijk} = \partial^{ijk} \psi$.
The sums are usually implicit in the Einstein notation: an index which appears twice in an expression, once as a superscript and once as a subscript,
is being summed upon from $1$ to $n$. The Legendre transform fits well with the Einstein notation,
thanks to identities such as
$$ \psi^{ijk}(x) = -\vphi_{ijk}(y) \quad \text{and} \quad \psi^{ij}_k(x) = -\vphi_{ij}^k(y), $$ where
expressions
involving $\psi$ are evaluated at the point $x \in L$
and expressions involving $\vphi$ are evaluated at the point $y = \nabla \psi(x) \in \Omega$.
Here, $(\nabla^2 \vphi)^{-1}(y) = (\vphi^{ij}(y))_{i,j=1,\ldots,n}$ and $\vphi_{ij}^k = \sum_{\ell} \vphi^{\ell k} \vphi_{ij\ell}$.
We may thus change variables  $y = \nabla \psi(x)$, and translate (\ref{eq_1121_}) to the equation: for any $y \in \Omega$ and $i=1,\ldots,n$,
\begin{equation} \left(y^j \vphi_j - \vphi \right) \vphi_{ik}^{k} = \left( n + 2 + \vphi_{jk}^{k} y^j \right) \vphi_i. \label{eq_331} \end{equation}
The function $\psi$ is smooth and strongly convex, hence the set $\{ x \in L \, ; \, \psi(x) \neq 0 \}$
is an open, dense set in $L$. Denote $U = \{ y \in \Omega \, ; \, \psi(\nabla \vphi(y)) \neq 0 \}$, an open, dense set in $\Omega$.
For any $y \in U$ we may define $$ A(y) = \frac{n + 2 + \vphi_{jk}^{k} y^j}{ \left( \sum_{\ell} y^\ell \vphi_\ell \right) - \vphi}. $$ Thus $\vphi_{ik}^k = A \vphi_i$
throughout the set $U$, according to (\ref{eq_331}). Moreover, the following holds in the  set $U$, for $i=1,\ldots,n$:
\begin{equation}  y^j \vphi_j \vphi_{ik}^{k} = A y^j \vphi_j \vphi_i = \vphi_{jk}^{k} y^j \vphi_i. \label{eq_333} \end{equation}
From (\ref{eq_331}) and (\ref{eq_333}), we have
\begin{equation} - \vphi \vphi_{ik}^{k} = ( n + 2 ) \vphi_i.
\label{eq_1050} \end{equation}
The validity of (\ref{eq_1050}) in the dense set $U \subseteq \Omega$ implies by continuity that (\ref{eq_1050})
holds true in the entire open set $\Omega$. By multiplying (\ref{eq_1050}) by $\vphi^{n+1} \cdot \det \nabla^2 \vphi$ we obtain
that in all of $\Omega$,
\begin{equation} \nabla ( \vphi^{n+2} \cdot \det \nabla^2 \vphi  ) = 0. \label{eq_339} \end{equation}
The set $\Omega$ is connected, being the image of the connected set $L$ under a smooth map.
Hence $\det \nabla^2 \vphi \cdot \vphi^{n+2} \equiv C$ in $\Omega$. This constant $C$ cannot be zero
according to (\ref{eq_1025}), because $\det \nabla^2 \vphi$ never vanishes in $\Omega$ and $\vphi$ is not the zero function.
This completes the verification of (\ref{eq_343}). We have thus proven one direction of (ii).
However, all of our manipulations in this proof are reversible: The validity of (\ref{eq_343})
implies the validity of (\ref{eq_1050}), which in turn leads to (\ref{eq_331}) and
eventually to (\ref{eq_1121}).
Hence (\ref{eq_343}) implies that $M$ is affinely-spherical with center at the origin.
\end{proof}

The following proposition is close to the original definition of affinely-spherical hypersurfaces
given by  Tzitz\'eica \cite{Tz1, Tz2}.

\begin{proposition} Let $M  \subset \RR^{n+1}$ be  a smooth, connected,
locally strongly-convex hypersurface. For $y \in M$
write $K_y > 0 $ for the Gauss curvature of $M$ at the point $y$ and denote
$$ \rho_y = \langle y, N_y \rangle $$
where $N_y \in \RR^{n+1}$ is the Euclidean unit normal to $M$ at the point $y$,  pointing to the concave
side of $M$. Then
$M$ is affinely-spherical with center at the origin if and only if there exists $C \in \RR \setminus \{ 0 \}$
such that $\rho_y^{n+2} / K_y = C$ for all $y \in M$. \label{prop_1348}
\end{proposition}

\begin{proof} See Nomizu and Sasaki \cite[Section II.5]{NS} for a proof of this proposition, or alternatively argue as follows:
Since $M$ is connected, it suffices to show that $M$ is affinely-spherical with center at the origin if and only if
the function $y \mapsto \rho_y^{n+2} / K_y$ is locally-constant in $M$ and it never vanishes.

\medskip Fix $y_0 \in M$. By applying a rotation in $\RR^{n+1}$, we may assume that
in a neighborhood of $y_0$, the hypersurface $M$ looks like the graph of a strongly-convex function. That is, we may assume that there exist an
open set $U \subseteq \RR^{n+1}$ with $y_0 \in U$, a convex, open set $L \subseteq \RR^n$ and a proper, convex
function $\psi: \RR^n \rightarrow \RR \cup \{ + \infty \}$ which is finite, smooth and strongly-convex in $L$,
such that
$$ M \cap U = \graph_L(\psi). $$
A standard exercise in differential geometry is to show that for any $x \in L$, at the point $y = (x, \psi(x))$,
\begin{equation}  \rho_y = \frac{\langle x, \nabla \psi(x) \rangle - \psi(x)}{\sqrt{1 + |\nabla \psi(x)|^2}},
\label{eq_1100} \end{equation}
and
\begin{equation}  K_y = \det \nabla^2 \psi(x) \cdot (1 + |\nabla \psi(x)|^2)^{-n/2-1}. \label{eq_1101}
\end{equation}
Denote $\vphi = \psi^*$. From (\ref{eq_1100}) and (\ref{eq_1101}) we obtain that
$$ \frac{\rho_y^{n+2}}{K_y} =\frac{ \left( \langle x, \nabla \psi(x) \rangle - \psi(x) \right)^{n+2}}{\det \nabla^2 \psi(x)} = \vphi^{n+2}(z) \cdot \det \nabla^2 \vphi(z) $$
where $z = \nabla \psi(x)$.  The desired conclusion now follows from Proposition \ref{prop_1147_}.
\end{proof}

\section{The polar affinely-spherical hypersurface}
\label{sec_5}
\setcounter{equation}{0}

In this section we prove Theorem \ref{thm1}. We begin with a variant of a construction in convexity
considered by Artstein-Avidan and Milman \cite{AM} and by Rockafellar \cite[Section 15]{roc}.
Fix a dimension $n$, and denote
$$ \cH^+ = \left \{ (x,t) \in \RR^n \times \RR \, ; \, t > 0 \right \} \subseteq \RR^{n+1},
\qquad \qquad \cH^- = \left \{ (x,t) \in \RR^n \times \RR \, ; \, t < 0 \right \} \subseteq \RR^{n+1}. $$
Consider the fractional-linear transformations $I^+: \cH^+ \rightarrow \cH^-$ and $I^-: \cH^- \rightarrow \cH^+$ defined via
$$ I^+(x, t) = \left( \frac{x}{t}, -\frac{1}{t} \right), \qquad I^-(y, s) = \left( -\frac{y}{s}, -\frac{1}{s} \right). $$
Then $I^+$ is a diffeomorphism whose inverse is $I^-$. A subset $V \subseteq \cH^{\pm}$ is a {\it relative half-space} if
$V =A \cap \cH^{\pm}$ where $A \subseteq \RR^{n+1}$ takes the form
$$  A = \left \{ (x,t) \in \RR^n \times \RR \, ; \, \langle x, \theta \rangle + b t + c \geq 0 \right \} \subseteq \RR^{n+1}
$$
for some $\theta \in \RR^n, b, c \in \RR$. Note that a relative half-space $V \subseteq \cH^{\pm}$ is a relatively-closed
subset of $\cH^{\pm}$. We say that a relative half-space $V \subseteq \cH^{\pm}$ is proper if $V$ and $\cH^{\pm} \setminus V$ are non-empty.

\begin{lemma} The maps $I^+$ and $I^-$  transform relative half-spaces to relative half-spaces.
\label{lem_1139}
\end{lemma}

\begin{proof} Let $\theta \in \RR^n, b, c \in \RR$. Then for any subset $V \subseteq \cH^+$,
\begin{align*} V =  \left \{ (x,t) \in \cH^{+} \, ; \, \langle x, \theta \rangle + b t + c \geq 0 \right \}
 \Longleftrightarrow  I^+(V) =
\left \{ (y,s) \in \cH^- \, ; \, \langle y, \theta \rangle - c s + b \geq  0  \right \}. \end{align*}
Hence $V \subseteq \cH^+$ is a relative half-space if and only if $I^+(V) \subseteq \cH^-$ is a relative half-space.
\end{proof}

Any relatively-closed subset $A \subseteq \cH^{\pm}$ which is convex
is the intersection of a family of relative half-spaces in $\cH^{\pm}$.
From Lemma \ref{lem_1139} we conclude
the following:

\begin{corollary} The maps $I^+$ and $I^-$  transform relatively-closed, convex sets to relatively-closed, convex sets.
\label{cor_136}
\end{corollary}

Similarly to  Rockafellar \cite[Section 15]{roc}, we say that the set $I^{\pm}(A)$ is the {\it obverse} of the set
$A \subseteq \cH^{\pm}$. See Figure \ref{fig3} for an example of a convex set and its obverse.
The {\it polar body} of a convex subset $S \subseteq \RR^{d}$ is defined via
$$ S^{\circ} = \left \{ x \in \RR^{d} \, ; \, \forall y \in S, \, \langle x, y \rangle \leq 1 \right \}. $$
The set $S^{\circ}$ is always convex, closed and  contains the origin. If $S \subseteq \RR^{d}$
is convex, closed and  contains the origin, then  $(S^{\circ})^{\circ} = S$.
For a subset $S \subseteq \RR^n$ and for a function $F: S \rightarrow \RR \cup \{+ \infty \}$ we write
$$ \epigraph_S(F) = \left \{ (x,t) \in S \times \RR \, ; \, F(x) \leq t \right \} \subseteq \RR^{n+1}. $$
When $S = \RR^n$ we abbreviate $\epigraph(F) = \epigraph_{\RR^n}(F)$.
Note that a function $F: \RR^n \rightarrow \RR \cup \{ + \infty \}$ is proper and convex if and only if
$\epigraph(F)$ is convex, closed and non-empty. The obverse operation interchanges between the Legendre transform
and the polarity transform:

\begin{proposition} Let $\vphi: \RR^n \rightarrow (0, +\infty]$ be a proper, convex function and denote $\psi = \vphi^*$. Then,
\begin{equation}  I^+( \epigraph(\vphi) ) = \epigraph(\psi)^{\circ} \cap \cH^-. \label{eq_932} \end{equation}
Moreover, if $\psi(0) < \infty$ then $\epigraph(\psi)^{\circ} \setminus \cH^- = \{ (x, 0) \, ; \, x \in \dom(\psi)^{\circ} \}
\subseteq \RR^n \times \RR$. \label{lem_1035}
\end{proposition}

\begin{proof} Denote $A = \epigraph(\vphi)$ and note that $A \subseteq \cH^+$ because $\vphi$ is positive.
For any $(y, -s) \in \cH^-$,
\begin{equation} (y,-s) \in I^+(A)
 \ \ \Longleftrightarrow \ \ (y/s,1/s) \in A \ \ \Longleftrightarrow \ \ \vphi(y/s) \leq 1/s.
\label{eq_322} \end{equation}
Recall that $(s \psi)^*(y) = s \vphi(y / s)$ for any $y \in \RR^n$ and $s > 0$. By (\ref{eq_322}),
for any $(y,-s) \in \cH^-$,
$$ (y,-s) \in I^+(A) \ \ \Longleftrightarrow \ \ (s \psi)^*(y) \leq 1
\ \ \Longleftrightarrow \ \ \forall x \in \dom(\psi), \ \langle x, y \rangle - s \psi(x) \leq 1. $$
Consequently,
\begin{align} \label{eq_505} I^+(A)  & = \left \{ (y,s) \in \RR^n \times \RR \, ; \, s < 0, \, \langle x, y \rangle + s \psi(y) \leq 1 \ \text{for all} \ x \in \dom(\psi) \right \} \\ & =
\left \{ (y,s) \in \RR^n \times \RR \, ; \, s < 0, \, \langle x, y  \rangle + ts \leq 1 \ \text{for all} \ (x,t) \in \epigraph(\psi)\right \}.
\nonumber
\end{align}
Hence $I^+(A) = \epigraph(\psi)^{\circ} \cap \cH^-$, and (\ref{eq_932}) is proven.
Next, assume that $\psi(0) < \infty$. Then $\epigraph(\psi)$ contains all points of the form $(0, t)$ for $t \geq \psi(0)$.
Therefore, for any $(y,s) \in \epigraph(\psi)^{\circ}$,
$$ \langle 0, y  \rangle + t s \leq 1 \qquad \text{for all} \ t \geq \psi(0), $$
and hence $s \leq 0$. We conclude that $\epigraph(\psi)^{\circ} \setminus \cH^- \subseteq \{ (y,0) \, ; \, y \in \RR^n \}$.
Consequently,
\begin{align*}
& \epigraph  (\psi)^{\circ}  \setminus \cH^-   = \left \{ (y,0) \, ; \, y \in \RR^n, \, \langle x, y  \rangle + t \cdot 0 \leq 1 \ \text{for all} \ (x,t) \in \epigraph(\psi)\right \}
\\ & = \left \{ (y,0) \, ; \, y \in \RR^n, \, \langle x, y  \rangle \leq 1 \ \text{for all} \ x \in \dom(\psi)\right \} = \{ (y, 0) \, ; \, y \in \dom(\psi)^{\circ} \}.
\tag*{\qedhere} \end{align*}
\end{proof}

\begin{figure}
\begin{center}
\begin{tikzpicture}[scale = 2]

\draw[ultra thick, fill=gray, domain=2:4] plot (\x, {1.6 * sqrt(0.1+(\x-3)*(\x-3) / 3)});
\draw[dotted, fill=gray, domain=1.8:2] plot (\x, {1.6 * sqrt(0.1+(\x-3)*(\x-3) / 3)});
\draw[dotted, fill=gray, domain=4:4.2] plot (\x, {1.6 * sqrt(0.1+(\x-3)*(\x-3) / 3)});

\draw[ultra thick, fill=gray, domain=-1.5:-0.5] plot (\x, {- 1.3* sqrt(0.25 - (\x+1)*(\x+1)) });
\draw[ultra thick, fill=gray, domain=-0.5:-1.5] plot (\x, {- 1.3* sqrt(0.25 - (\x+1)*(\x+1)) });
\draw[ultra thick, gray] (-1.4, -0.01) -- (-0.6, -0.01);


\draw [<->] (2,0) -- (4,0);
\draw [<->] (3,-1) -- (3,1);

\draw [<->] (-2,0) -- (0,0);
\draw [<->] (-1,-1) -- (-1,1);

\end{tikzpicture}
\caption{A semi-circle and its obverse, which is a branch of a hyperbola. \label{fig3}}
\end{center}
\end{figure}
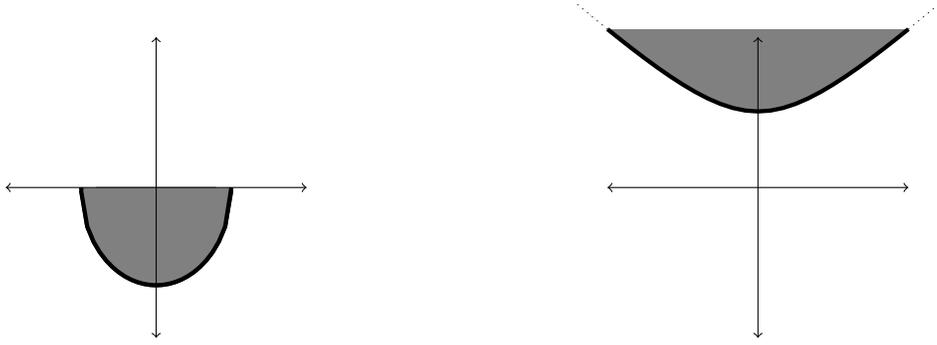

For a subset $A \subseteq \cH^{\pm} \subseteq \RR^{n+1}$ we write $\overline{A} \subseteq \RR^{n+1}$ and $\partial A \subseteq \RR^{n+1}$
for the usual closure and boundary of the set $A$, viewed as a subset of $\RR^{n+1}$.
Similarly, when $A \susbeteq \cH^{\pm} \subseteq \RR^{n+1}$ is convex, we write $A^{\circ}$ for its polar body, where again $A$
is  viewed as a convex subset of $\RR^{n+1}$. When  $A \subseteq \cH^{\pm}$ is relatively-closed, its closure $\overline{A}$ is contained in $\overline{\cH^{\pm}}$, and
$\overline{A} \cap \cH^{\pm} = A$.
Note that the relative boundary of a subset  $A \subseteq \cH^{\pm}$ equals $(\partial A) \cap \cH^{\pm}$.

\begin{lemma}
The two diffeomorphisms $I^{\pm}$ transform smooth,  connected, locally strongly-convex hypersurfaces to smooth,  connected, locally strongly-convex hypersurfaces.
\label{cor_1026}
\end{lemma}

\begin{proof}
Let $M \susbeteq \cH^{\pm}$ be a smooth, connected hypersurface.
A locally-supporting relative half-space at the point $y \in M$ is a proper, relative half-space
$A \subseteq \cH^{\pm}$ with $y \in \partial A$ such that $A \supseteq M \cap U$
 for some open neighborhood $U \subseteq \cH^{\pm}$ of the point $y$.

\medskip A smooth, connected hypersurface $M \subseteq \cH^{\pm}$ is locally strongly-convex
if and only if for any $y \in M$ there is a unique locally-supporting-relative-half-space at the point $y$,
which varies smoothly in $y \in M$ and without critical points.

\medskip The diffeomorphisms $I^{\pm}$
induce a diffeomorphism between the space of proper, relative half-spaces of $\cH^+$ and the space of proper, relative half-spaces of $\cH^-$,
as we see from the proof of Lemma \ref{lem_1139}.  Thus, if $M \subseteq \cH^{\pm}$ is a smooth, connected, locally strongly-convex hypersurface
then the same holds for $I^{\pm}(M)$. The lemma is thus proven.
\end{proof}

We say that a subset $A \subseteq \cH^{\pm}$ is bounded from below if there exists $(x_0, t_0) \in \cH^{\pm}$ such that
$$ t > t_0 \qquad \qquad \qquad \text{for all} \ (x,t) \in A. $$
It is easy to verify that if $A \subseteq \cH^{\pm}$ is bounded from below, then its obverse is also bounded from below.

\begin{lemma} Let $L \subseteq \RR^n$ be a bounded, open, convex set containing the origin.
Let $B \subseteq \cH^-$ be a relatively-closed, convex set that is bounded from below. Assume that the set $(\partial B) \cap \cH^-$ is
a smooth, connected, locally strongly-convex hypersurface, while
$ (\partial B) \setminus \cH^- = \{ (x, 0)  \, ; \, x \in L^{\circ} \}$.

\medskip Then there exists a proper, convex function $\psi: \RR^n \rightarrow \RR \cup \{ + \infty \}$
with $\overline{\dom(\psi)} = \overline{L}$, that is smooth and strongly-convex in $L$,
with $\nabla \psi(L) = \RR^n, \psi(0) < 0$ and $ \overline{B} = \epigraph(\psi)^{\circ}$.
Moreover, $I^-(B) = \epigraph(\vphi)$ where $\vphi = \psi^*$.
\label{prop_1446}
\end{lemma}

\begin{proof} Since $B \subseteq \cH^-$, for any $(x, t) \in \RR^n \times \RR$ and $ r > 0$,
$$ (x, t) \in B^{\circ} \quad \Longrightarrow \quad (x, t + r) \in B^{\circ}. $$
Therefore the closed set $B^{\circ}$ satisfies $B^{\circ} = \epigraph(\psi)$ where $\psi: \RR^n \rightarrow \RR \cup \{ + \infty \}$
is defined via $$ \psi(x) = \inf \{ t \in \RR \, ; \, (x,t) \in B^{\circ} \}. $$ Here, $\inf \emptyset = +\infty$. Since $B^{\circ} \subseteq \RR^{n+1}$
is closed, convex and it contains the origin, the function $\psi$ is necessarily proper and convex. The set $\overline{B}$ is closed, convex
and it contains the origin, as follows from our assumptions. Since $\overline{B}^{\circ} = B^{\circ} = \epigraph(\psi)$ while $B \subseteq \cH^-$ is relatively-closed,
\begin{equation}  \overline{B} = \epigraph(\psi)^{\circ} \qquad \text{and} \qquad B = \overline{B} \cap \cH^- = \epigraph(\psi)^{\circ} \cap \cH^-. \label{eq_1143}
\end{equation}
The set $B \subseteq \cH^-$ is bounded from below, hence there exists $t_0 < 0$ such that $t > t_0$ for all $(x,t) \in B$. Therefore $(0, 1 / t_0) \in B^{\circ}$ and thus $\psi(0) < 0$.
Denote $\vphi = \psi^*$. Then $\vphi: \RR^n \rightarrow (0, +\infty]$ is proper and convex. By (\ref{eq_1143}) and Proposition \ref{lem_1035},
\begin{equation}  A := I^-(B) = I^-( \epigraph(\psi)^{\circ} \cap \cH^- ) = \epigraph(\vphi) \label{eq_1806A} \end{equation}
and moreover,
\begin{equation}  (\partial B) \setminus \cH^- = \overline{B} \setminus \cH^- = \epigraph(\psi)^{\circ} \setminus \cH^-  = \{ (x,0) \, ; \, x \in \dom(\psi)^{\circ} \}.
\label{eq_1253} \end{equation}
However, $ (\partial B) \setminus \cH^- = \{ (x, 0)  \, ; \, x \in L^{\circ} \}$ according to our assumptions. From (\ref{eq_1253}) we thus deduce
that  $L^{\circ} = \dom(\psi)^{\circ}$ and $\overline{L} = \overline{\dom(\psi)}$.
Since $\dom(\psi) \subseteq \RR^n$ is bounded and $\vphi = \psi^*$, necessarily \begin{equation} \dom(\vphi) = \RR^n \label{eq_1258} \end{equation} by \cite[Corollary 13.3.3]{roc}.
The map $\cI^-$ is a homeomorphism, and hence it transforms the relative-boundary
of $B \subseteq \cH^-$, which is the set $(\partial B) \cap \cH^-$, to the relative-boundary of $A \subseteq \cH^+$, which is the set $(\partial A) \cap \cH^+$.
Since the relative-boundary $(\partial B) \cap \cH^-$ is a smooth, connected, locally strongly-convex hypersurface,
Lemma \ref{cor_1026} implies that also the hypersurface $$ (\partial A) \cap \cH^+ = I^-( (\partial B) \cap \cH^-) $$
is smooth, connected and locally strongly-convex. Since $\inf \vphi = -\psi(0) > 0$, the relations (\ref{eq_1806A}) and (\ref{eq_1258}) imply
that  $$ (\partial A) \cap \cH^+ = \partial A  = \graph_{\RR^n} (\vphi). $$
 Hence $\graph_{\RR^n}(\vphi)$ is a smooth, connected, locally strongly-convex hypersurface. Consequently $\vphi: \RR^n \rightarrow \RR$ is smooth and strongly-convex.
This implies that the set
$\nabla \vphi(\RR^n)$ is the interior of $\dom(\vphi^*)$ (see, e.g., \cite[Theorem 26.5]{roc} or \cite[Section 1.2]{gro}).
We conclude that  $\nabla \vphi(\RR^n) = L$, and \cite[Theorem 26.5]{roc} shows that the function $\psi = \vphi^*$ is smooth and strongly-convex in $L$
with  $\nabla \psi(L) = \RR^n$. We have thus verified all of the conclusions of the lemma.
\end{proof}

There are two convex epigraphs that are associated with the convex set $B \subseteq \cH^-$
from Lemma \ref{prop_1446}: the obverse of $B$ is $\epigraph(\vphi)$ while the polar of $B$ is $\epigraph(\psi)$. We may 
think about this triplet of convex sets as three different ``coordinate systems'' for describing
the affine hemisphere equation. We will shortly see that $\partial B \cap \cH^-$ is an affine hemisphere centered at the origin
if and only if $\epigraph_L(\psi)$ is affinely-spherical with center at the origin,
which happens if and only if $\vphi$ satisfies $\det \nabla^2 \vphi = C / \vphi^{n+2}$.
Recall that for a smooth hypersurface $M \subseteq \RR^{n+1}$ and $y \in M$, we view the tangent space $T_y M$
as an $n$-dimensional linear subspace of $\RR^{n+1}$.

\begin{definition} Let $M  \subseteq \RR^{n+1}$ be  a smooth,  connected,
locally strongly-convex hypersurface. Assume that $y \not \in T_y M$ for all $y \in M$.
For $y \in M$ define the vector $\nu_y \in \RR^{n+1}$ via the requirements that
$$ \langle \nu_y, y \rangle = 1, \qquad \nu_y \perp T_y M. $$
We refer to $\nu: M \rightarrow \RR^{n+1}$ as the ``polarity map''. We define the ``polar hypersurface'' $M^*$ via
$$ M^{*} := \nu(M) = \left \{ \nu_y \, ; \, y \in M \right \}. $$
\label{def_1502}
\end{definition}

What is the relation between polar hypersurfaces and polar bodies?
If $S \subseteq \RR^{n+1}$ is a convex set
and if  $M \subseteq \partial S$ is a smooth,  connected,
locally strongly-convex hypersurface for which the polarity map is well-defined, then $M^{*} \subseteq \partial S^{\circ}$. Thus, Definition \ref{def_1502}
provides a local version of the theory of convex duality: a piece of the boundary of $S$ is polar to a certain piece of the boundary of $ S^{\circ}$.

\medskip Suppose that $M  \subseteq \RR^{n+1}$ is a smooth,  connected,
locally strongly-convex hypersurface such that $y \not \in T_y M$ for all $y \in M$.
It is well-known that $M^{*}$ is always a smooth,  connected,
locally strongly-convex hypersurface such that  $y \not \in T_y M^{*}$ for all $y \in M^{*}$.
Furthermore, the polarity map $\nu: M \rightarrow M^{*}$
is a diffeomorphism, and its inverse is the polarity map associated with $M^*$. In particular, $(M^{*})^{*} = M$.

\begin{lemma} Let $L \subseteq \RR^n$ be an open, bounded, convex set containing the origin.
Let $\psi: \RR^n \rightarrow \RR \cup \{ + \infty \}$ be a proper, convex function with $\psi(0) < 0$ such that $\overline{L} = \overline{\dom(\psi)}$.
Assume that $\psi$ is smooth and strongly-convex in $L$ with $\nabla \psi(L) = \RR^n$.
Denote $$ M = \graph_L(\psi) \qquad \text{and} \qquad \tilde{K} = \epigraph(\psi)^{\circ}. $$ Then $M^*$ is well-defined, the convex set $\tilde{K}$ is compact
with $\dim(\tilde{K}) = (n+1)$, and
\begin{equation} (\partial \tilde{K}) \cap \cH^- = M^* \qquad \text{while} \qquad (\partial \tilde{K}) \setminus \cH^- = \left \{ (x,0) \, ; \, x \in L^{\circ} \right \}.
\label{eq_1347} \end{equation}
\label{lem_1018}
\end{lemma}

\begin{proof}
Define $\vphi = \psi^*$. Since $\nabla \psi(L) = \RR^n$,
necessarily $\dom(\vphi) = \RR^n$ by \cite[Corollary 13.3.3]{roc}.
 Since $\psi(0) < 0$, the function  $\vphi: \RR^n \rightarrow \RR$ is  positive and convex.
  Denote  $A = \epigraph(\vphi) \subseteq \cH^+$ and $B = \tilde{K} \cap \cH^-$. By Proposition \ref{lem_1035},
\begin{equation} B = \tilde{K} \cap \cH^- = \epigraph(\psi)^{\circ} \cap \cH^- = I^+(\epigraph(\vphi)) = I^+(A). \label{eq_432} \end{equation}
Since $\vphi: \RR^n \rightarrow \RR$ is convex and positive, we
may assert that  $\partial A \cap \cH^+ = \partial A = \graph_{\RR^n}(\vphi)$. Consequently
\begin{equation} \partial \tilde{K} \cap \cH^- = \partial B \cap \cH^- = I^+(\partial A \cap \cH^+) = I^+(\graph_{\RR^n}(\vphi)).
\label{eq_433} \end{equation}
Since $\psi$ is smooth in $L$, the identity $\psi(x) + \vphi(\nabla \psi(x)) = \langle x, \nabla \psi(x) \rangle$ holds for all $x \in L$.
The fact that $\nabla \psi(L) = \RR^n$ thus implies
\begin{equation}
\graph_{\RR^n}(\vphi) = \left \{ \left(\nabla \psi(x), \langle x, \nabla \psi(x) \rangle - \psi(x)  \right) \in \RR^n \times \RR \, ; \, x \in L \right \}.
\label{eq_435} \end{equation}
Note that $\langle x, \nabla \psi(x) \rangle - \psi(x) = \vphi(\nabla \psi(x)) > 0$ for all $x \in L$,
and hence $\nu_y$ is indeed well-defined.
It follows from Definition \ref{def_1502} that for $x \in L$ and $y = (x, \psi(x)) \in \graph_L(\psi)$,
\begin{equation} \nu_y =  \frac{\left(\nabla \psi(x), -1 \right)}{\langle x, \nabla \psi(x) \rangle - \psi(x)} = I^+ \left \{ \, \left(\nabla \psi(x), \langle x, \nabla \psi(x) \rangle - \psi(x)  \right) \, \right\}. \label{eq_439}
\end{equation}
Since $M = \graph_L(\psi)$ and $M^* = \nu(M)$, by (\ref{eq_433}), (\ref{eq_435}) and (\ref{eq_439}),
\begin{equation}  M^* = \nu(\graph_L(\psi)) = I^+(\graph_{\RR^n}(\vphi)) = \partial \tilde{K} \cap \cH^-. \label{eq_917A}
\end{equation}
Proposition \ref{lem_1035} shows
that $\tilde{K} = \epigraph(\psi)^{\circ} \subseteq \overline{\cH^-}$. In fact,
according to Proposition \ref{lem_1035},
\begin{equation} (\partial \tilde{K}) \setminus \cH^- =  \tilde{K} \setminus \cH^- = \{ (x, 0)  \, ; \, x \in \dom(\psi)^{\circ} \} = \{ (x, 0)  \, ; \, x \in L^{\circ} \}. \label{eq_1031}
\end{equation}
Now (\ref{eq_1347}) follows from (\ref{eq_917A}) and (\ref{eq_1031}).
It follows from (\ref{eq_1347}) that $\dim(\tilde{K}) = n+1$, since the convex set $\tilde{K}$ affinely-spans the hyperplane $\partial \cH^-$ while it also contains
points outside this hyperplane. Moreover, since $0 \in L$ and $\psi(0) < 0$, the convex set
$\epigraph(\psi)$ contains a neighborhood of the origin in $\RR^{n+1}$. Therefore the closed set $\tilde{K} = \epigraph(\psi)^{\circ}$ is bounded,
and hence it is compact.
\end{proof}

Recall from Proposition \ref{prop_1348} that $N_y$ is the Euclidean unit normal to $M$ at the point $y$ that is pointing to the concave
side of $M$. Recall also that we denote $\rho_y = \langle N_y, y \rangle$. It follows from Definition \ref{def_1502} that
if $\rho_y \neq 0$ for all $y \in M$ then the polarity map is well-defined, and
\begin{equation} \nu_y = \frac{N_y}{\rho_y} \qquad \qquad \text{for all} \ y \in M. \label{eq_1518}
\end{equation}
The map $N: M \rightarrow S^{n}$ is the Gauss map associated with $M$, and we see that the polarity map
is proportional to the Gauss map.
We define the {\it cone measure} on a smooth hypersurface $M \subseteq \RR^{n+1}$ to be the measure $\mu_M$ supported
on $M$ whose density with respect to the surface area measure on $M$ is the function $y \mapsto |\rho_y|/(n+1)$.
The reason for the term ``cone measure'' is that for any Borel subset $S \subseteq M$ that does not contain
two distinct points on the same ray from the origin,
$$ \mu_M(S) = \Vol_{n+1} \left( \left \{ t x \, ; \, 0 \leq t \leq 1, \, x \in S \right \} \right ). $$

\begin{proposition}
Let $M  \subseteq \RR^{n+1}$ be a smooth, connected, locally strongly-convex hypersurface. Then $M$ is affinely-spherical with center at the origin
if and only if the following holds: The polarity map $\nu: M \rightarrow M^*$ is well-defined, and it pushes forward the cone measure $\mu_M$ to a measure proportional
to the cone measure $\mu_{M^*}$. \label{prop_1015}
\end{proposition}

\begin{proof} If $M$ is affinely-spherical with center at the origin then the polarity map of $M$ is well-defined,
since $\rho_y \neq 0$ for all $y \in M$ according to Proposition \ref{prop_1348}. For $y \in M$ let $S_y: T_y M\rightarrow T_y M$ be the shape operator
associated with the Euclidean unit normal $N$. Then $\det(S_y)$ is the Gauss curvature $K_y > 0$.
For any vector field $X$ tangent to $M$ we have
\begin{equation} D_X \nu = D_X \left( N / \rho \right) = \frac{S(X)}{\rho} - \frac{D_X \rho}{\rho^2} N, \label{eq_1207}
\end{equation}
where $D_X \nu \in \RR^{n+1}$ is the derivative of $\nu$ in the direction of $X$.
Write $D \nu: T M \rightarrow T M^{*}$ for the differential of the smooth polarity map $\nu$.
Then for any $y \in M$, the map $(D \nu)_y$ is a linear map from the tangent space $T_y \cM = \nu_y^{\perp}$ to the tangent space $T_{\nu_y} M^{*} = y^{\perp}$.
Here, $y^{\perp}$ is the hyperplane orthogonal to $y$ in $\RR^{n+1}$.
From (\ref{eq_1207}), for any $y \in M$ and $u \in T_y M$,
\begin{equation} S_y(u) = \rho_y \cdot Proj_{\nu_y^{\perp}} \left( (D \nu)_y(u) \right), \label{eq_1213} \end{equation}
where $Proj_{\nu_y^{\perp}}$ is the orthogonal projection operator onto $\nu_y^{\perp}$ in $\RR^{n+1}$.
The operator $Proj_{\nu_y^{\perp}}: y^{\perp} \rightarrow \nu_y^{\perp}$ distorts $n$-dimensional
volumes by a factor of $|\langle y, \nu_y \rangle| / (|y| |\nu_y|)$.
The linear map $(D \nu)_y: \nu_y^{\perp} \rightarrow y^{\perp}$ distorts volumes by a
factor of $|\det (D \nu)_y|$. Hence, by (\ref{eq_1213}),
for any $y \in M$,
\begin{equation} K_y = \det(S_y) =
|\rho_y|^n \cdot \frac{|\langle y, \nu_y \rangle|}{|y| |\nu_y|} \cdot |\det (D \nu)_y|  =
 \frac{|\det (D \nu)_y|}{|y| |\nu_y|^{n+1}}, \label{eq_1218}
\end{equation}
where we used (\ref{eq_1518}) in the last passage. In fact, according to (\ref{eq_1518}), the cone measure $\mu_M$ has density $y \mapsto 1 / ( (n+1) |\nu_y|)$ with respect to the surface area measure on $M$. Denote by $\theta$ the measure on $M$ whose density with respect to the surface area measure is $K_y |\nu_y|^{n+1} / (n+1)$.

\medskip Recalling that the polarity map of $M^*$ is inverse to that of $M$, we deduce from (\ref{eq_1218})
that $\nu$ pushes forward $\theta$ to the cone measure $\mu_{M^*}$. Consequently, $\nu$ pushes forward $\mu_M$ to a measure proportional
to $\mu_{M^*}$ if and only if $\theta$ is proportional to $\mu_M$, i.e., if and only if there exists $C > 0$ such that
\begin{equation} K_y |\nu_y|^{n+1} / (n+1) = C / ((n+1) |\nu_y|) \qquad \qquad \text{for all} \ y \in M.
\label{eq_1014} \end{equation}
Recall that $1 / |\nu_y| = |\rho_y|$, and that $\nu$ and $\rho$ are continuous
in the connected manifold $M$. By Proposition \ref{prop_1348},
the hypersurface $M$ is affinely-spherical with center at the origin if and only if there exists $C > 0$
such that (\ref{eq_1014}) holds true. This completes the proof.
\end{proof}

Since the polarity map of $M^*$ is the inverse to the polarity map of $M$,
Proposition \ref{prop_1015} has the following well-known corollary:

\begin{corollary} Let $M \subseteq \RR^{n+1}$ be an affinely-spherical hypersurface with center at the origin.
Then the polar hypersurface $M^{*}$ is well-defined, and it is again affinely-spherical with center at the origin. \label{dual}
\end{corollary}

\begin{theorem} Let $L \subseteq \RR^n$ be an open, bounded, convex set containing the origin. Then the following are equivalent:
\begin{enumerate}
\item[(i)] The barycenter of $L$ lies at the origin.
\item[(ii)] There exists a proper, convex function $\psi: \RR^n \rightarrow \RR \cup \{ + \infty \}$ with $\overline{\dom(\psi)} = \overline{L}$
such that $\graph_L(\psi)$ is affinely-spherical with center at the origin, and such that $\psi$ is smooth and strongly-convex in $L$ with $\nabla \psi(L) = \RR^n$ and $\psi(0) < 0$.
\end{enumerate}
Moreover, assuming (i) or (ii), the function $\psi$ from (ii) is uniquely determined up to a multiplication
by a positive scalar $\lambda > 0$ and an addition of a linear function $\ell(x) = \langle x, v \rangle$, for some $v \in \RR^n$.
\label{thm4}
\end{theorem}

\begin{proof} Assume (i). According to Theorem \ref{thm3}, there exists a smooth, positive, convex function $\vphi: \RR^n \rightarrow \RR$ with
$\nabla \vphi(\RR^n) = L$ such that
\begin{equation}
\det \nabla^2 \vphi = \frac{C}{\vphi^{n+2}} \qquad  \text{in} \ \RR^n, \label{eq_927}
\end{equation}
for some constant $C > 0$. Denote $\psi = \vphi^*$. From  \cite[Theorem 26.5]{roc} we know that $\overline{\dom(\psi)} = \overline{L}$
and that $\psi$ is smooth and strongly convex in $L$ with $\nabla \psi(L) = \RR^n$.
According to Proposition \ref{prop_1147_}, equation (\ref{eq_927}) implies that $\graph_L(\psi)$ is affinely-spherical with center at the origin.
The infimum of $\vphi$ is attained and is positive because $0 \in L$. Hence $\psi(0) < 0$, and we have verified all conclusions in (ii).

\medskip Next, assume (ii) and let us prove (i). Denote $\vphi = \psi^*$. Since $\overline{L} = \overline{\dom(\psi)}$ is a bounded set, necessarily
$\dom(\vphi) = \RR^n$ by \cite[Corollary 13.3.3]{roc}. Since $\psi$ is smooth and strongly-convex in $L$
with $\nabla \psi(L) = \RR^n$ and $\psi(0) < 0$, necessarily $\vphi$ is a positive, smooth, strongly-convex function in $\RR^n$ with $\nabla \vphi(\RR^n) = L$.
Since $\graph_L(\psi)$
is affinely-spherical with center at the origin, Proposition \ref{prop_1147_} shows that (\ref{eq_927}) holds true.
Theorem \ref{thm3} now implies (i). Moreover,  Theorem \ref{thm3} states that $\vphi$ is uniquely determined up to translations and dilations, implying
that $\psi$ is determined up to the transformation described above.
\end{proof}

Let $K \subseteq \RR^n$ be an $n$-dimensional, non-empty, bounded, convex set.
The {\it Santal\'o point} of $K$ is the unique point $z(K) \in \RR^n$ such that
$$ \Vol_n( (K - z(K))^{\circ} ) = \inf_{z \in \RR^n} \Vol_n (K - z)^{\circ} $$
where $K - z = \{ x - z \, ; \, x \in K \}$.
The Santal\'o point of $K$ is well-defined and it belongs to the interior of $K$, see \cite[Section 7.4]{sch}.
The Santal\'o point of $K$ satisfies $z(K) = 0$ if and only if the barycenter of $K^{\circ}$ is well-defined and it lies at the origin.
The Santal\'o point is affinely-invariant: for any invertible, affine transformation $T: \RR^n \rightarrow \RR^n$ we know that
$z(T(K)) = T(z(K))$.
Hence the Santal\'o point is well-defined for any non-empty, bounded, convex set embedded in some finite-dimensional
real linear space.

\begin{proof}[Proof of the existence part of Theorem \ref{thm1}] By applying an affine transformation in $\RR^{n+1}$,
we may assume that the Santal\'o point of $K$ lies at the origin, and that
$$ K \subseteq \{ (x, 0) \, ; \, x \in \RR^n \}. $$
Write $K_1 \subseteq \RR^n$ for the interior of the set $\left \{ x \in \RR^n \, ; \, (x,0) \in K \right \}$.
Then $K_1 \subseteq \RR^n$ is an open, convex set whose Santal\'o point lies at the origin.
Hence $K_1^{\circ} \subseteq \RR^n$
is a compact, convex set containing zero in its interior such that the barycenter of $K_1^{\circ}$ lies at the origin. Write $L \subseteq \RR^n$ for the interior of $K_1^{\circ}$.
It follows from Theorem \ref{thm4} that
there exists a proper, convex function $\psi: \RR^n \rightarrow \RR \cup \{ + \infty \}$ with $\overline{\dom(\psi)} = \overline{L}$
such that $$ M := \graph_{L}(\psi) $$ is affinely-spherical with center at the origin. Moreover, $\nabla \psi(L) = \RR^n$ and $\psi(0) < 0$.
Denote
$$ \tilde{K} = \epigraph(\psi)^{\circ}. $$
According to Corollary \ref{dual}, the hypersurface $M^*$ is affinely-spherical with center at the origin.
Furthermore, Lemma \ref{lem_1018} shows that $\tilde{K} \subseteq \RR^{n+1}$ is an $(n+1)$-dimensional, compact convex set and
$$ M^* = (\partial \tilde{K}) \cap \cH^-  \qquad \text{while} \qquad (\partial \tilde{K}) \setminus \cH^- = L^{\circ} \times \{ 0 \} = K. $$
Consequently $M^* \susbeteq \cH^-$ does not intersect the hyperplane $\partial \cH^-$ that contains $K$, while $\partial \tilde{K} = M^* \cup K$.
According to Definition \ref{def_1251}, the hypersurface  $M^*$ is an affine hemisphere with anchor $K$, which is centered at the Santal\'o point of $K$.
\end{proof}

\begin{proposition}
 Let $L \subseteq \RR^n$ be a bounded, open, convex set containing the origin.
Let $M \subseteq \cH^-$ be an affine hemisphere with anchor $L^{\circ} \times \{ 0 \} \subseteq \RR^n \times \RR = \RR^{n+1}$ and center at the origin. Then $M^*$ is well-defined, and there exists a function $\psi$ as in Theorem \ref{thm4}(ii) such that $ M^* = \graph_L(\psi)$. \label{prop_1153}
\end{proposition}

\begin{proof} The hypersurface  $M \subseteq \cH^-$ is an affine hemisphere with anchor $K = L^{\circ} \times \{ 0 \}$ which is centered  at the origin.
Let $\tilde{K}$ be as in Definition \ref{def_1251}.
Denote $B = \tilde{K} \cap \cH^-$ which is a convex, relatively-closed subset of  $\cH^-$ with $\overline{B} = \tilde{K}$.
The convex set $B$ is bounded from below
in $\cH^-$ since $\tilde{K}$ is compact. Moreover, by Definition \ref{def_1251} the set
\begin{equation} M = (\partial \tilde{K}) \cap \cH^- = (\partial B) \cap \cH^- \label{eq_1447} \end{equation}
is a smooth, connected, locally strongly-convex hypersurface. Additionally, it follows from Definition \ref{def_1251} that
\begin{equation} (\partial B) \setminus \cH^- = (\partial \tilde{K}) \setminus \cH^- = K = L^{\circ} \times \{ 0  \}. \label{eq_917} \end{equation}
Thus the  relatively-closed, convex set $B \subseteq \cH^-$ satisfies all of the requirements of Lemma \ref{prop_1446}.
From the conclusion of Lemma \ref{prop_1446},
there exists a proper, convex function $\psi: \RR^n \rightarrow \RR \cup \{+ \infty \}$  such that
\begin{equation} \epigraph(\psi)^{\circ} = \overline{B} = \tilde{K} \label{eq_1147} \end{equation}
and such that $\psi(0) < 0, \overline{\dom(\psi)} = \overline{L}$ while
$\psi$ is smooth and strongly-convex in $L$ with $\nabla \psi(L) = \RR^n$.
Thanks to (\ref{eq_1447}) and (\ref{eq_1147}), Lemma \ref{lem_1018} shows that
$$ \graph_L(\psi) = M^*. $$
Since $M$ is affinely-spherical with center at the origin, Corollary \ref{dual} implies that $\graph_L(\psi)$ is also affinely-spherical with center at the origin.
Hence the function $\psi$ satisfies all of the conditions of Theorem \ref{thm4}(ii), and the proposition is proven.
\end{proof}

\begin{proof}[Proof of the uniqueness part of Theorem \ref{thm1}]
Suppose that $M$ is an affine hemisphere with anchor $K$,
and let $\tilde{K}$ be as in Definition \ref{def_1251}.
By applying an affine transformation in $\RR^{n+1}$,
we may assume that $M$ is affinely-spherical with center at the origin, and that
\begin{equation}  K \subseteq \{ (x, 0) \, ; \, x \in \RR^n \} \qquad \text{while} \qquad \tilde{K} \subseteq \overline{\cH^-}. \label{eq_1154} \end{equation}
Definition \ref{def_1251} implies that the origin belongs to the relative interior of the $n$-dimensional, compact, convex set $K$.
Hence there exists a bounded, open, convex set $L \subseteq \RR^n$  containing the origin such that $K = L^{\circ} \times \{ 0 \}$.
From (\ref{eq_1154}) and Definition \ref{def_1251} we conclude that $M = \partial \tilde{K} \cap \cH^- \subseteq \cH^-$.
Proposition \ref{prop_1153} shows that $M^* = \graph_L(\psi)$ for a certain convex function $\psi: \RR^n \rightarrow \RR \cup \{ + \infty \}$
satisfying the requirements of Theorem \ref{thm4}(ii).

\medskip Theorem \ref{thm4} now implies that the barycenter of $L$ lies at the origin, and hence the affine hemisphere $M$ is centered at the Santal\'o
point of $K$. According to Theorem \ref{thm4}, the function $\psi$ is uniquely determined by $L$, up to a multiplication by a positive scalar and an addition of a linear function.
It thus follows that the affine hemisphere $M = \graph_L(\psi)^*$ with anchor $L^{\circ} \times \{0\}$ is uniquely determined by $L$, up to
a linear transformation. Therefore $M$ is determined by $K$ up to an affine transformation, and the proof is complete.
\end{proof}

\begin{remark} {\rm Let $M$ be an affine hemisphere in $\RR^{n+1}$ with center at the origin and anchor $K \susbeteq \RR^n \times \{ 0 \}$.
Let $\tilde{K} \subseteq \RR^n \times [0, \infty)$ be the convex body from Definition \ref{def_1251},
so that $\partial \tilde{K} = M \cup K$. For $(x,t) \in \RR^n \times [0, \infty)$ set
$$ \| (x,t) \|_{\tilde{K}} = \inf \left \{ \lambda > 0 \, ; \, (x,t) / \lambda \in \tilde{K} \right \}, $$
the Minkowski functional of $\tilde{K}$. Denote also $F(x,t) = \| (x,t) \|_{\tilde{K}}^2/2$. Since the origin belongs
to the relative interior of $K$, the function $F$ is a finite, $2$-homogenous, convex
function in the half-space $(x,t) \in \RR^n \times [0, \infty)$. Note that the closure of the affine hemisphere $M$ is a level set of the function $F$.
It was noted by Bo Berndtsson that the function $F$ satisfies
\begin{equation}  \left \{ \begin{array}{rccl} \det \nabla^2 F(x,t) & = & C  & \text{for} \ (x,t) \in \RR^n \times (0, \infty) \\ F(x,0) & = & \| x \|_K^2/2  & \text{for} \ x \in \RR^n \end{array} \right. \label{eq_1313} \end{equation}
where $C > 0$ is a positive constant and $\| x \|_K = \inf \{ \lambda > 0 \, ; \, x / \lambda \in K \}$ is the Minkowski functional of $K$.
Thus $F$ solves the parabolic affine sphere equation $\det \nabla^2 F \equiv Const$ in a half-space, with boundary values that are $2$-homogenous and convex.
In order to prove the equation in (\ref{eq_1313}), we argue as follows: The map $\nabla F$ restricted to $M$ is precisely the polarity map
of the affine hemisphere $M$. Since $\nabla F$ is $1$-homogenous, for any measurable subset $A \subseteq M$ and $0 < \alpha < \beta$,
\begin{equation}  \{ \nabla F(t y) \, ; \, y \in A, \, \alpha < t < \beta \} = \left \{ t z \, ; \, z \in \nu(A), \, \alpha < t < \beta \right \} \label{eq_1334}
\end{equation}
where $\nu: M \rightarrow M^*$ is the polarity map associated with $M$. Proposition \ref{prop_1015} states that $\nu$ pushes forward the cone volume measure
on $M$ to a constant multiple of the cone volume measure on $M^*$. It thus follows from (\ref{eq_1334}) that $\nabla F$
pushes forward the Lebesgue measure on $\tilde{K}$ to a constant multiple of the Lebesgue measure on $\{ t y \, ; \, y \in M^*, t \in [0,1] \}$.
Therefore the Jacobian of the map $y \mapsto \nabla F(y)$ has a constant determinant, and (\ref{eq_1313}) is proven.
} \label{rem_1338} \end{remark}

{
}


\end{document}